\documentclass[english]{article} % la classe du document
\usepackage[utf8]{inputenc} % encodage des caractères
\usepackage[T1]{fontenc} % encodage de la fonte
\usepackage[a4paper]{geometry} % la gestion de la géométrie de la page
\usepackage{amsmath, amssymb, mathtools, amsthm}
\usepackage{mathtools} % pour tous les ams[...]
\usepackage{mathrsfs}  % \mathscr
\usepackage{graphicx} % pour \includegraphics{monJPG}
\usepackage{xcolor}   % pour gérer les couleurs
\usepackage[english]{babel} % gestion des langues
\usepackage{bbm}
\usepackage[title]{appendix}
\newcommand{\eps}{\varepsilon}
\usepackage{caption}%    Personnalisation des légendes des flottants
\usepackage{subcaption}% Sous figures / sous tableaux
\usepackage{comment}
\usepackage{geometry}
\usepackage{hyperref} % les liens hypertextes

\newcommand \commentout[1] {}

\newcommand{\R}{\mathbb{R}}
\newcommand{\N}{\mathbb{N}}

\newcommand {\sg} {\sigma}

\newcommand {\Chi} {{\bf \raise 2pt \hbox{$\chi$}} }

\newcommand {\sgn} { {\rm sgn} }
\newcommand {\Div}  { {\rm div} }
\newcommand {\f}   {\frac}
\newcommand {\p}   {\partial}

\newcommand*{\dd}{\mathop{\kern0pt\mathrm{d}}\!{}}

\DeclarePairedDelimiter{\abs}{\lvert}{\rvert}
\DeclarePairedDelimiter{\norm}{\lVert}{\rVert}

\theoremstyle{plain}
\newtheorem*{thm*}{Theorem}
\newtheorem{thm}{Theorem}[section]

\newtheorem{proposition}[thm]{Proposition}

\theoremstyle{remark}
\newtheorem{remark}[thm]{\bf Remark}

\newcommand{\ie}{\emph{i.e.}\;}

\newcommand{\eg}{\emph{e.g.}\;}
\newcommand{\etal}{\emph{et al.}\;}

\newcommand{\beq}{\begin{equation}}
\newcommand{\eeq}{\end{equation}}
\newcommand{\bea} {\begin{array}{rl}}
\newcommand{\eea} {\end{array}}
\newcommand{\bepa}{\left\{ \begin{array}{l}}
\newcommand{\eepa} {\end{array}\right.}

\numberwithin{equation}{section}

\title{Degenerate Cahn-Hilliard and incompressible limit of a Keller-Segel model}

\author{Charles Elbar\thanks{Sorbonne Universit\'{e}, CNRS, Universit\'{e} de Paris, Inria, Laboratoire Jacques-Louis Lions (LJLL), F-75005 Paris, France } \thanks{email: charles.elbar@sorbonne-universite.fr}
\and Beno\^it Perthame\footnotemark[1]  \thanks{email: benoit.perthame@sorbonne-universite.fr}
\and Alexandre Poulain\thanks{Department of Numerical Analysis and Scientific Computing, Simula Research Laboratory, Oslo, Norway} \thanks{email: poulain@simula.no}
}

\date{\today}

\begin{document}

\maketitle

\begin{abstract}
    The Keller-Segel model is a well-known system representing chemotaxis in living organisms. We study the convergence of a generalized nonlinear variant of the Keller-Segel to the degenerate Cahn-Hilliard system. This analysis is made possible from the observation that the Keller-Segel system is equivalent to a relaxed version of the Cahn-Hilliard system. Furthermore, this latter equivalent system has an interesting application in the modelling of living tissues. Indeed, compressible and incompressible porous medium type equations are widely used to describe the mechanical properties of living tissues. The relaxed degenerate Cahn-Hilliard system, can be viewed as a compressible living tissue model for which the movement is driven by Darcy's law and takes into account the effects of the viscosity as well as surface tension at the surface of the tissue.  
    %Compared to classical versions of the Cahn-Hilliard model, here the pressure function replaces the interaction potential. Indeed, the movement of cells are driven by Darcy's law and the model comprises surface tension through a bi-laplacian term. 
    We study the convergence of the Keller-Segel system to the Cahn-Hilliard equation and some of the analytical properties of the model such as the incompressible limit of our model. 
    Our analysis relies on a priori estimates, compactness properties, and on the equivalence between the Keller-Segel system and the relaxed degenerate Cahn-Hilliard system.  
\end{abstract}
\vskip .7cm

\noindent{\makebox[1in]\hrulefill}\newline
2010 \textit{Mathematics Subject Classification.}  35B40; 35B45; 35G20; 35Q92
\newline\textit{Keywords and phrases.} Degenerate Cahn-Hilliard equation;  Asymptotic analysis; Keller-Segel system; Incompressible limit; Hele-Shaw equations; Surface tension

\section{Introduction}

%\AP{Changes for the introduction: present the convergence from Keller-Segel to Cahn-Hilliard without giving the intermediate step that is the relaxed CH. Then present the two theorems (KS to CH and incompressible limit). Then, give the equivalence between KS and RDCH. Motivations unchanged. Existence Section at the end . Create a section called "Formal derivation" in which we have the Energies for RDCH and the convergence theorem KS to CH (or $\sg \to 0$. Then, incompressible  limit. Finally Existence and uniqueness.}

When describing living tissue dynamics, the most relevant Hele-Shaw free boundary models come with a surface tension term, \cite{Escher-solutions,Friedman}. In the regime of a smooth free boundary, these problems are well established and obtained as the sharp-interface limits of Cahn-Hilliard equations \cite{Alikakos_convergence}. However, it is an open question to prove such a derivation in the global regime of weak solutions. Our goals are to establish the incompressible limit departing from the Relaxed Cahn-Hilliard system (RCH in short), and to make the link with a specific Generalized Keller-Segel (GKS) type system, namely 
%
%We study the convergence of a nonlinear Keller-Segel type model set in a smooth open bounded domain $\Omega\subset \R^d$ in dimension $d\le 3$ and ${\Omega_{T}=(0,T)\times\Omega}$ for $T>0$ to the generalized degenerate Cahn-Hilliard model. The Keller-Segel system reads
\begin{align}
\partial_{t} n-\f\delta{2\sg}\Delta n^2 + \f{\delta}{\sg}\Div\left(n\nabla w\right) &= nG\left(\f\delta\sg\left(n-w\right)\right),\quad &\text{in}\quad &(0,+\infty)\times\Omega,\label{eq:KS1}\\
-\sigma\Delta w +  \f\sg\delta w^\gamma + w  &= n,&\quad \text{in}\quad &(0,+\infty)\times\Omega,\label{eq:KS2}
 \end{align}
where  $\Omega\subset \R^d$ is a smooth open bounded domain. We are going to prove that as $\sg $ vanishes, its limit is the  Degenerate Cahn-Hilliard (DCH) model 
 \begin{align}
\partial_{t} n-\Div(n\nabla \mu)=nG(\mu),\quad \text{in}\quad &(0,+\infty)\times\Omega,\label{eq:CH1-nonrelax}\\
\mu =n^{\gamma}-\delta\Delta n,\quad \text{in}\quad &(0,+\infty)\times\Omega \label{eq:CH2-nonrelax},
\end{align}
where $\mu$ is formed of a repulsion potential $n^\gamma$ and a surface tension potential $-\delta\Delta n$ (see further explanations below). In this work, $n^{\gamma}$ actually stands for $\max(0,n)^{\gamma}$ but since the solutions are nonnegative, we keep the notation $n^{\gamma}$ for sake of clarity. The growth (or source) term $G(\mu)$ takes into account death and birth of cells.

It is crucial for the analysis of the convergence of the GKS system to the DCH equation to remark that the nonlinear system~\eqref{eq:KS1}--\eqref{eq:KS2} is equivalent to a relaxed version of the Cahn-Hilliard model (see \cite{perthame_poulain}).
Indeed, defining  
\begin{equation}
\label{eq:defofw}
\mu = \f\delta\sigma\left(n-w\right),\quad  w= n-\frac{\sigma}{\delta}\mu, \qquad p=w^\gamma, \quad \mu =p- \delta \Delta w,
\end{equation}
System~\eqref{eq:KS1}--\eqref{eq:KS2} can be rewritten in the form
\begin{align}
\partial_{t} n-\Div(n\nabla \mu)=nG(\mu),\quad \text{in}\quad &(0,+\infty)\times\Omega,\label{eq:CH1}\\
-\sg \Delta \mu + \mu =\left(n-\frac{\sigma}{\delta}\mu \right)^{\gamma}-\delta\Delta n,\quad \text{in}\quad &(0,+\infty)\times\Omega \label{eq:CH2}.
\end{align}
Although they are equivalent,  Systems~\eqref{eq:KS1}--\eqref{eq:KS2} and \eqref{eq:CH1}--\eqref{eq:CH2} carry different informations. For instance, since we consider non-negative densities $n \geq 0$, we also have
\[
w \geq 0, \qquad \mu \leq \frac {\delta} {\sigma} n.
\]
The Keller-Segel formulation of the system also provides higher regularity of the solutions. In the following, we combine analytical methods better adapted to each of these two equivalent formulations 
%to show that in the limit $\sg\to 0$, the nonlinear GKS model is equivalent to the degenerate Cahn-Hilliard type model~\eqref{eq:CH1-nonrelax}--\eqref{eq:CH2-nonrelax}. The equivalence between the nonlinear GKS model and the relaxed degenerate Cahn-Hilliard model allows us also
in order to pass to the incompressible limit $\gamma\to \infty$ and obtain
\begin{equation}
\partial_{t}n_{\sigma,\infty}-\Div(n_{\sigma,\infty}\nabla\mu_{\sigma,\infty})=n_{\sigma,\infty}G(\mu_{\sigma,\infty}),
\label{IL:n}
\end{equation}
\begin{equation}\left\{ \begin{array}{l}
\mu_{\sigma,\infty}=p_{\sigma,\infty}-\delta\Delta w_{\sigma,\infty},
\\[5pt]
-\sigma\Delta w_{\sigma,\infty}+\frac{\sigma}{\delta}p_{\sigma,\infty}+w_{\sigma,\infty}=n_{\sigma,\infty},\quad w_{\sigma,\infty}=n_{\sigma,\infty}-\frac{\sigma}{\delta}\mu_{\sigma,\infty}.
\end{array} \right.
\label{IL:mu} \end{equation}
In this 'stiff pressure limit', the system~\eqref{IL:n}--\eqref{IL:mu} has three unknowns $n_{\sigma,\infty},w_{\sigma,\infty}$ (or $\mu_{\sigma,\infty}$) and
$p_{\sigma,\infty}$ and is completed with a type of incompressibility condition
\begin{equation}
    p_{\sigma,\infty}(w_{\sigma,\infty}-1)=0.
\label{IL:p}    
\end{equation}

For these systems, we use the zero mass and energy flux boundary conditions
\begin{equation}
\label{boundary}
n\frac{\partial\mu}{\partial\nu}= \frac{\partial w}{\partial\nu}=0	\quad\text{on}\quad (0,\infty)\times\partial\Omega,
\end{equation}
where $\nu$ is the outward normal vector to the boundary of $\Omega$, and the initial condition
\begin{equation}
%\begin{cases}
    n(0,\cdot) = n^0 \in H^1(\Omega)\cap L^{\infty}(\Omega),\qquad \text{and}\qquad  0\le n^0\le 1 \quad \text{a.e.}
%    \\[5pt]
%    \mu(0,x) = \mu^0_{\sigma,\gamma}(x) \quad \text{solution of } -\sigma\Delta\mu^0_{\sigma,\gamma} +\mu^0_{\sigma,\gamma} =(n^0_{\sigma,\gamma}-\frac{\sigma}{\delta}\mu^0_{\sigma,\gamma})^{\gamma}-\delta\Delta n^0_{\sigma,\gamma},`
%     \\[5pt]
%    w(0,x) =n(0,x)\mu(0,x) .
%    \\ w(0,x) = n^0(x) - \frac{\sigma}{\delta} \mu^0(x).
%    \end{cases}
    \label{initcond}
\end{equation}
We also need assumptions on the pressure-dependent proliferation rate of the cells 
%is assumed to be  and Equation~\eqref{eq:CH1} takes into account this effect through the growth rate $G(\mu)$. Furthermore, we assume that the growth rate $G(\cdot)$ satisfies
\begin{equation}\label{eq:assumption G} 
 G\in C(\R; \R), \qquad    \sup_{\mu \in \R} (1+|\mu|) |G(\mu)|<+\infty.
\end{equation}
For instance we can suppose that there exists $\mu_{H}$ such that $G(\mu)=0$ for $|\mu|>\mu_{H}$. In that case, the value $\mu_H$ can be viewed as the \textit{homeostatic pressure} which is the lowest level of pressure that prevents cell multiplication due to contact-inhibition. For the pressure law exponent, we assume
\begin{equation} 
    \gamma > 1.
\label{eq:assumption gamma}
\end{equation}

Using assumptions~\eqref{eq:assumption G} the total mass of the system is bounded. More precisely, we find a constant $C_{T}$ such that for all $t\in(0,T)$ and for all $\sigma,\gamma$, it holds for all $T>0$
\begin{equation}
\label{PW}
\frac{1}{|\Omega|}\int_{\Omega}n_{\sigma,\gamma}(t,\cdot)\le C_{T},\quad \frac{1}{|\Omega|}\int_{\Omega}w_{\sigma,\gamma}(t,\cdot)\le C_{T}.
\end{equation}
This allows us to use the Poincaré-Wirtinger inequality.

System~\eqref{eq:CH1}--\eqref{eq:CH2} comes with the energy and entropy respectively defined by 
\begin{equation*}
\mathcal{E}[n,\mu]=\int_{\Omega}\frac{(n-\frac{\sigma}{\delta}\mu)^{\gamma+1}}{\gamma+1}+\frac{\delta}{2}\Big|\nabla \Big(n-\frac{\sigma}{\delta}\mu\Big)\Big|^{2}+\frac{\sigma}{\delta}\frac{|\mu|^{2}}{2},
\end{equation*}
\begin{equation*}
\Phi[n]=\int_{\Omega}n\log n.
\end{equation*}
They formally satisfy, as already used in~\cite{perthame_poulain}, the relations
\begin{equation*}
\frac{d}{dt}\mathcal{E}[n,\mu]=-\int_{\Omega}n|\nabla\mu|^{2}+\int_{\Omega}n\mu G(\mu),
\end{equation*}
\begin{equation*}
\begin{aligned}
\frac{d}{dt}\Phi[n]=-\int_{\Omega}\delta\Big|\Delta\Big(n-\frac{\sigma}{\delta}\mu\Big)\Big|^{2}&+\frac{\sigma}{\delta}|\nabla\mu|^{2}\\
&+\gamma\Big(n-\frac{\sigma}{\delta}\mu\Big)^{\gamma-1}\Big|\nabla \Big(n-\frac{\sigma}{\delta}\mu\Big)\Big|^{2}+\int_{\Omega}n G(\mu)(\log(n)+1).
\end{aligned}
\end{equation*}
From \eqref{eq:defofw}, they also provide informations on $w$.
The purpose of this work is to establish estimates which allow us to study the convergence of the nonlinear GKS model to the DCH model, and to analyze the incompressible limit $\gamma \to \infty$.

\paragraph{Modelling of living tissues and assumptions}
Our study lies among other models which represent biological phenomena that we present here.

System~\eqref{eq:CH1}--\eqref{eq:CH2} models the movement and proliferation of a population of cells constituting a biological tissue and driven by the combined effects of the pressure, the surface tension occurring at the surface of the tissue, as well as its viscosity.
As in the context of the modelling of diphasic fluid, in System~\eqref{eq:CH1}--\eqref{eq:CH2}, $n$ is the order parameter, \ie the relative cell density $n=n_1/(n_1+n_2)$. The unknown $\mu$ is a quantity related to the effective pressure and is also used to relax the fourth order Cahn-Hilliard type equation in a system of two-second order equations. Following the Cahn-Hilliard terminology or Mechanobiology, we refer to $\mu$ indistinctly as the chemical potential or as the effective pressure.
The equation for the effective pressure (\ie Equation~\eqref{eq:CH2}) contains the effects of both the pressure, through the term $\left(n-\f{\sg}{\delta} \mu \right)^\gamma$ with $\gamma > 1$ that controls the stiffness of the pressure law, and surface tension by $- \delta \Delta n$, where $\sqrt\delta$ is the width of the interface in which partial mixing of the two components $n_1$, $n_2$ occurs. Equation~\eqref{eq:CH2} also contains a diffusive term $-\sg\Delta \mu$ used to relax the Cahn-Hilliard system~\eqref{eq:CH1}--\eqref{eq:CH2} as in \cite{perthame_poulain}. This relaxation term can also be interpreted as the effect of viscosity. 

%We now present the assumptions that are used in our analysis.
%First, our problem is endowed by zero-flux boundary conditions
%\begin{equation}\label{boundary}
%\frac{\partial n}{\partial\nu}=\frac{\partial\mu}{\partial\nu}= \frac{\partial w}{\partial\nu}=0	\quad\text{on}\quad %(0,\infty)\times\partial\Omega,\end{equation}
%where $\nu$ is the outward normal vector to the boundary $\partial\Omega$. 
%The systems are also closed by the initial conditions

%To summarize, in Model~\eqref{eq:CH1}--\eqref{eq:CH2}, the movement of cells is driven by the combination of the effects of pressure, surface tension and viscosity. 

%Interestingly, we notice that using the notation 
%\begin{equation}
%\label{eq:defofw}
%w=n-\f\sg\delta \mu,
%\end{equation}
%the system~\eqref{eq:CH1}--\eqref{eq:CH2} can be rewritten in the form
%\begin{align}
%\partial_{t} n-\f\delta{2\sg}\Delta n^2 + \f{\delta}{\sg}\Div\left(n\nabla w\right) = nG\left(\f\delta\sg\left(n-w\right)\right),\quad \text{in}\quad &(0,+\infty)\times\Omega,\label{eq:KS1}\\
%-\sigma\Delta w +  \f\sg\delta w^\gamma + w  = n,\quad \text{in}\quad &(0,+\infty)\times\Omega.\label{eq:KS2}
% \end{align}
%We here easily recognize a nonlinear Keller-Segel model. 
%\paragraph{Motivations and previous works}
Modelling tissue growth and understanding the dynamics of cells have been the center of many research pieces in the past decade. Initiated by Greenspan~\cite{Greenspan}, general mechanical models of tumor growth \cite{byrne_modelling_2004,Friedman,Lowengrub-bridging} have been proposed and used the internal pressure as the main effect that drives the movement and proliferation of cells. The prototypical example of a mechanical living tissue model is
\begin{equation}
    \p_t n = \Div\left(n\nabla p \right) + nG(p),\quad p=P_\gamma(n):=\f{\gamma}{\gamma-1} n^{\gamma-1}, 
    \label{eq:general-mod}
\end{equation}
in which $p$ is the pressure and $n$ the cell density. 
In this kind of model, Darcy's law of movement is used to reflects the porous media formed by the extra-cellular matrix and the tendency of cells to move away from regions of high compression. The dependency on the pressure of the growth function has also been used to model the sensitivity of tissue proliferation to compression. The behavior and stability of monotone traveling waves has been studied in~\cite{dalibard:hal-03300624}.
Interestingly, Perthame \etal~\cite{Perthame-Hele-Shaw} have shown that in the incompressible limit (\ie $\gamma\to \infty$), solutions of Model~\eqref{eq:general-mod} converge to a limit $(n_\infty,p_\infty)$ solution of a free boundary limit problem of Hele-Shaw type for which the speed of the free boundary is given by the normal component of $p_\infty$. In this limit, the solution of Equation~\eqref{eq:general-mod} organizes into 2 regions: $\Omega(t)$ in which the pressure is positive (corresponding to the tissue) and outside this zone where $p=0$. Furthermore, the free Boundary problem is supplemented by a complementary equation that indicates that the pressure satisfies
\begin{equation}
    -\Delta p_\infty = G(p_\infty),\quad \text{in}\quad \Omega(t),\quad \text{or similarly}\quad p_\infty(\Delta p_\infty +  G(p_\infty)) = 0 \quad \text{a.e. in } \Omega.
    \label{eq:complem}
\end{equation}
However, the crucial role of the cell-cell adhesion at the surface of the tissue is not retrieved at the limit. Indeed, as pointed by Lowengrub \etal~\cite{Lowengrub-bridging}, the velocity of the free surface should depend on its local curvature denoted by $\kappa$. 

Thus, multiple variants of the general model~\eqref{eq:general-mod} have been suggested to consider other physical effects in mechanical models of tissue growth. 
The addition of the effect of viscosity in the model has been made to represent the friction between cells~\cite{Basan,Bittig_2008} through the use of Stokes' or Brinkman's law (see \cite{Allaire} for a rigorous derivation of Brinkman's law in inhomogeneous materials). However, as pointed out by Perthame and Vauchelet~\cite{Perthame-incompressible-visco}, Brinkman's law leads to a simpler version of the model and, therefore, is a preferential choice for its mathematical analysis. 
Adding viscosity through the use of Brinkman's law leads to the model  
\begin{equation}
\begin{cases}
    \p_t n = \Div\left(n\nabla \mu \right) + nG(p), \quad &\text{in}\quad (0,+\infty)\times\Omega,\\
    -\sg \Delta \mu + \mu = p,\quad &\text{in}\quad (0,+\infty)\times\Omega.
    \end{cases}
    \label{eq:visco-living}
\end{equation}
The incompressible limit of this system also yields the complementary relation (see~\cite{Perthame-incompressible-visco})
\[
    p_\infty(p_\infty - \mu_\infty - \sigma G(p_\infty)) = 0,\quad \text{a.e. in } \Omega.
\]
In the incompressible limit, notable changes compared to the system with Darcy's law are found. First the previous complementary relation is different compared to Equation~\eqref{eq:complem}, and the pressure $p_\infty$ in the limit is discontinuous, \ie there is a jump of the pressure located at the surface of $\Omega(t)$.
However, the pressure jump is related to the potential $\mu$ and not to the local curvature of the free boundary $\p \Omega(t)$. The authors already indicated that a possible explanation to this is that the previous model does not include the effect of surface tension.
Indeed, classical solutions for the Hele-Shaw problem with the addition of the effect of surface tension leads to the free boundary problem (see \eg \cite{Escher-solutions})
\begin{equation}
    \begin{cases}
        -\Delta \mu = 0\quad \text{in}\quad \Omega \setminus \p\Omega(t),\\
        \mu = \sigma \kappa \quad \text{on} \quad \p\Omega(t).
    \end{cases}
\end{equation}
This correct Hele-Shaw limit has been formally obtained as the sharp-interface asymptotic model of the Cahn-Hilliard equation~\cite{Alikakos_convergence}.

Cahn-Hilliard type models are widely used nowadays to represent living tissues and in particular tumors~\cite{wise_three-dimensional_2008,Frieboes-2010-CH}. Being of fourth-order, the Cahn-Hilliard equation is often rewritten in a system of two second-order equations, \ie 
\begin{equation}
    \p_t n = \Div\left(n\nabla \left(\psi^\prime(n)-\delta \Delta n \right)\right) \to \begin{cases}
        \p_t n &= \Div\left(n\nabla \mu \right),\\
        \mu &= -\delta \Delta n +\psi^\prime(n),
    \end{cases}
    \label{eq:Cahn-Hilliard}
\end{equation}
where $n$ is the concentration of a phase and $\mu$ is called the chemical potential in material sciences but is often used as an effective pressure for living tissues (see \cite{chatelain_2011,BenAmar_C_F,Agosti-CH-2017,chatelain_morphological_2011}). Also, the interaction potential $\psi(n)$ contained in this effective pressure term comprises the effects of attraction and repulsion between cells. The physically relevant form of this potential is a double-well logarithmic potential and is often approximated by a smooth polynomial function. However, recent studies show that for the modelling of living tissues and for the particular application where only one of the component of the mixture experiences attractive and repulsive forces, a single-well logarithmic potential is more relevant \cite{byrne_modelling_2004}. In our particular application, the potential $\psi^\prime(n)$ is only related to the modelling of pressure and, therefore, the final model includes only repulsive forces. 

We emphasize that multiple variants of the Cahn-Hilliard model have been studied to represent living tissues and tumors including passive transport by Darcy's Law \cite{Garcke-2016-CH-darcy,Garcke-2018-multi-CH-darcy,Frigeri-CH-Darcy} and with the effect of viscosity \cite{ebenbeck_cahn-hilliard-brinkman_2018,Ebenbeck-Brinkman}. The existence of weak solutions for degenerate mobility has been first studied by Elliott and Garcke in \cite{Garcke-CH-deg}. The case of a smooth potential and mobility was treated by Dai and Du in \cite{Dai-Du}. In the definition of their weak solutions, the authors did not identify the potential $\mu$ except in the zone where the density does not vanishes. Indeed this difficulty comes from the energy which provides the bound $\|n^{1/2}\nabla\mu\|_{L^{2}(\Omega_{T})}\le C$. The term $\Div(n\nabla\mu)$ in the weak solutions is then treated as $\Div(n^{1/2}\zeta)$ where $\zeta$ is an $L^{2}$ vector field that can be identified in the zone $n\ne 0$. 
For our system, we identify the potential $\mu$ by considering a weaker type of solutions.

\commentout{

For the sake of simplicity, that is often needed to retain the capacity to analyze rigorously a model, we combine the two modelling frameworks provided by the general living tissue model~\eqref{eq:visco-living} and the Cahn-Hilliard model~\eqref{eq:Cahn-Hilliard} to design a system of equations that includes the effect of both viscosity and surface tension and study its incompressible limit. At the end, our model~\eqref{eq:CH1}--\eqref{eq:CH2} is a modified {relaxed degenerate Cahn-Hilliard model} introduced by Perthame and Poulain in ~\cite{perthame_poulain} and is equivalent to a nonlinear variant of the GKS system.

}

\commentout{
Therefore, our idea is to add surface tension in the living tissue model of \cite{Perthame-incompressible-visco} that already includes viscosity, and study its incompressible limit.  

- Cahn-Hilliard for living tissue \cite{wise_three-dimensional_2008,Frieboes-2010-CH}, compressible Cahn-Hilliard ? \cite{Lowengrub-quasi-incompressible-1998,abels_diffuse_2008}
- CH models with interpretation of chemical potential as pressure \cite{chatelain_2011,BenAmar_C_F,Agosti-CH-2017,chatelain_morphological_2011}
- Cahn-Hilliard Darcy \cite{Garcke-2016-CH-darcy,Garcke-2018-multi-CH-darcy,Frigeri-CH-Darcy}
- Cahn-Hilliard Brinkman (to add viscosity) \cite{ebenbeck_cahn-hilliard-brinkman_2018,Ebenbeck-Brinkman}
- Relaxation of Cahn-Hilliard \cite{perthame_poulain}

- Incompressible limit in compressible living tissue model \cite{Perthame-Hele-Shaw,Perthame-incompressible-visco,Debiec-Brinkman}
}

\paragraph{Contents of the paper}
\begin{sloppypar}
We present and analyze the convergence of the nonlinear GKS model to the DCH model (corresponding to the study of the vanishing viscosity limit, \ie $\sg\to 0$). We also study rigorously the incompressible limit, \ie $\gamma \to \infty$ for this new living tissue model.
\end{sloppypar}
Section~\ref{sec:structure} contains  the proof that the nonlinear GKS system~\eqref{eq:KS1}--\eqref{eq:KS2} converges to the DCH model~~\eqref{eq:CH1-nonrelax}--\eqref{eq:CH2-nonrelax}. This analysis is made possible from the fact that the nonlinear GKS model is equivalent to the RCH model and from the use of standard compactness properties. 
To the best of our knowledge, this is the first time that a Keller-Segel system is interpreted as a Cahn-Hilliard  system.
The study of the incompressible limit is the purpose of Section~\ref{sec:inpressible-limit}. Since uniform bounds in $\gamma$ are required to use standard compactness properties and pass to the limit, we provide new estimates relying on results obtained for the non-linear GKS model. 
Section~\ref{section:eps} contains a proof of the existence of weak solution for the system~\eqref{eq:KS1}--\eqref{eq:KS2} and can be read first for the reader's convenience. The proof follows the lines of the works \cite{Garcke-CH-deg,perthame_poulain}. Indeed, using a regularization of the mobility, we apply a Galerkin approximation to show the existence of weak solutions to the regularized model. 
The main novelty in the proof of this latter result is the need of strong convergence for the chemical potential which is required to pass to the limit for the source term.
Then, a priori estimates on the regularized model give us sufficient control to use standard compactness results and pass to the limit of the regularization, hence obtaining the existence of weak solutions to the RCH-DKS model. Uniqueness of the weak solution is also shown in the case where there is no proliferation term. 

Therefore, our analysis relies on a combination of results obtained for the Cahn-Hilliard model as well as the non-linear GKS limit.  The addition of the regularization term in the Cahn-Hilliard model opens a new angle of attack to solve new problems and to find the correct Hele-Shaw limit for a simple living tissue model.

\section{From the GKS to the DCH system ($\sigma\to 0$)}
\label{sec:structure}
%-----------------------------------------------------------

We consider weak solutions of System~\eqref{eq:KS1}--\eqref{eq:KS2} as built in Section~\ref{sec:existence} and prove that weak solutions of the DCH equation can be obtained as the limit of the GKS system. In the definition of the weak solutions, there is not enough regularity on $\mu$ to treat the term $-\Div(n\nabla\mu)$ as $\int n\nabla\mu\nabla\chi$ where $\chi$ is a test function. We need to rely on another integration by parts based on the definition of $\mu$ and Equation~\eqref{weakform3}.

Weak solutions of the DCH equation~\eqref{eq:CH1-nonrelax}--\eqref{eq:CH2-nonrelax} satisfy both forms
\begin{align}
\int_{0}^{T}\langle \chi,\partial_{t}n\rangle&=\int_{\Omega_{T}}\Big[\frac{\gamma}{\gamma+1} n^{\gamma+1}-\delta n\Delta n-\delta\frac{|\nabla n|^{2}}{2}\Big]\Delta\chi+\delta\sum_{i=1}^{d}\partial_{i}n\nabla n\cdot\nabla\partial_{i}\chi+n G_{\infty}\chi,
\label{CH:weak1}\\
\int_{0}^{T}\langle \chi,\partial_{t}n\rangle&=\int_{\Omega_{T}}\frac{\gamma}{\gamma+1} n^{\gamma+1}\Delta\chi-\delta \Delta n(\nabla n\cdot\nabla\chi +n\Delta\chi)+n G_{\infty}\chi,
\label{CH:weak1_2}
\end{align}
for all $\chi\in L^{\infty}(0,T,W^{2,\infty}(\Omega))$ with $\nabla\chi\cdot\nu=0$, and
\begin{equation}
\int_{0}^{T}\langle \chi,\partial_{t}n\rangle=- \int_{\Omega_{T}} n \partial_{t} \chi - \int_{\Omega}\chi(t=0) n^0,
\label{CH:weak2}
\end{equation}
when, additionally, $\chi \in C^1([0,T]\times \Omega)$ and $\chi (T)=0$. The source term can be identified as $G_{\infty}=G(\mu)$ in dimension $d\le 4$. The weak formulation~\eqref{CH:weak1} comes from Equation~\eqref{weakform3}. The forms~\eqref{CH:weak1} and~\eqref{CH:weak1_2} are equivalent with the formula
\begin{equation*}
\frac{1}{2}\nabla(|\nabla f|^{2})=\Div(\nabla f\otimes\nabla f)-\Delta f\nabla f. 
\end{equation*}
  
We have the following convergence theorem.  
%--------------------------
\begin{thm}[Convergence of the GKS system to the DCH eq.]
\label{convKSCH}
\begin{sloppypar}
 Assume \eqref{initcond}--\eqref{eq:assumption gamma}. The weak solutions $(n_{\sigma}, \mu_{\sigma})$ of the RCH-DKS model converge to $(n,\mu)$  weak solution of~\eqref{eq:CH1-nonrelax}--\eqref{eq:CH2-nonrelax} in the sense defined by~\eqref{CH:weak1}--\eqref{CH:weak2}. They satisfy the regularity estimates  $n\in L^{\infty}(0,T,H^{1}(\Omega))\cap L^{2}(0,T,H^{2}(\Omega))$, $\partial_{t}n\in (L^{4}(0,T,W^{1,s}(\Omega)))'$ where $s$ is defined in Proposition~\ref{compactnessh2}, $\mu\in L^{2}(0,T,L^{\frac{\gamma+1}{\gamma}}(\Omega))\cap L^{\frac{\gamma+1}{\gamma}}(0,T,L^{\frac{d(\gamma+1)}{(d-2)\gamma}}(\Omega))$.
 \end{sloppypar}

In dimension $d\le4$, $\mu$ belongs to $L^{2}(\Omega)
$ and for all $\chi\in L^{\infty}(0,T,W^{2,\infty}(\Omega))$ with $\nabla\chi\cdot\nu=0$, Equation~\eqref{CH:weak1} can be written as
\begin{align}
\int_{0}^{T}\langle \chi,\partial_{t}n\rangle&=\int_{\Omega_{T}}\mu(\nabla n\cdot\nabla\chi+n\Delta\chi)+nG(\mu)
\chi.\label{CH:weak3}
\end{align}
with $\mu=n^{\gamma}-\delta\Delta n$ a.e. in $\Omega$.
\end{thm}

The rest of this section is devoted to prove this theorem. 

%-------------------------------
\subsection{A priori estimates}
Since the two systems~\eqref{eq:KS1}--\eqref{eq:KS2} and~\eqref{eq:CH1}--\eqref{eq:CH2} are equivalent, all estimates proved for a system apply to the other. Based on the construction in Section~\ref{sec:existence}, the energy and entropy structure of the RCH model provide us with the following bounds. Assuming~\eqref{initcond}--\eqref{eq:assumption gamma}, there is a constant $C$ independent of $\gamma,\sigma$, such that, for all $T\ge 0$, 
\begin{align}
\frac{\sigma}{\delta}\int_{\Omega_{T}}|\nabla\mu|^{2}\le C,\quad \frac{\sigma}{\delta}\int_{\Omega}|\mu(t)|^{2}&\le C\quad  \forall t\in (0,T),  \label{ine111}
\\
\delta\int_{\Omega}\Big|\nabla \Big(n(t)-\frac{\sigma}{\delta}\mu(t)\Big)\Big|^{2}&\le C, \quad \forall t\in (0,T),\label{ine311}
\\
\delta\int_{\Omega_{T}}\Big|\Delta\Big(n-\frac{\sigma}{\delta}\mu\Big)\Big|^{2}&\le C, \label{ine411}\\
\int_{\Omega_{T}}n|\nabla\mu|^{2}&\le C, \label{ine511}\\
\gamma\int_{\Omega_{T}}\Big(n-\frac{\sigma}{\delta}\mu\Big)^{\gamma-1}\Big|\nabla \Big(n-\frac{\sigma}{\delta}\mu\Big)\Big|^{2}&\le C, \label{ine611}\\
\int_{\Omega}\frac{(n(t)-\frac{\sigma}{\delta}\mu(t))^{\gamma+1}}{\gamma+1}&\le C, \quad \forall t\in (0,T).  \label{ine711}
\end{align}

Next, we deduce further a priori estimates which are enough to pass to the limit as $\sg$ vanishes.

%To complete this section about {a priori} estimates, we prove a $(L^{4}(0,T,W^{1,s}(\Omega)))'$ bounds on the derivative of the density. 

%The $L^{2}(0,T,H^{-1}(\Omega))$ bound cannot hold as the mobility is not uniformly bounded in $\eps$.

\begin{proposition}[A priori bounds]
\label{compactnessh2}
Assuming~\eqref{initcond}--\eqref{eq:assumption gamma}, there is a constant $C(T)$ independent of $\gamma$ and $\sigma$ such that the a priori estimates hold
\begin{equation*}
\|n\|_{L^{2}(0,T,H^{1}(\Omega))}\le C,
\qquad
\|\partial_{t}n\|_{(L^{4}(0,T,W^{1,s}(\Omega)))'}\le C,
\end{equation*}
with $s=2$ for $d=1$, $s>2$ for $d=2$ and $s=\frac{4d}{d+2}$ otherwise. Furthermore, we have 
\begin{equation*}
\| w \|_{L^{\infty}(0,T,L^{\gamma+1}(\Omega))}\le C, \qquad 
\f 1 \gamma \| w \|^{\gamma+1}_{L^{\gamma+1}(0,T,L^{\f{d(\gamma+1)}{d-2}}(\Omega))}\le C, \qquad 
\f 1 \gamma \| \mu \|^\f{\gamma+1}{\gamma}_{L^\f{\gamma+1}\gamma (0,T,L^{q}(\Omega))}\le C,
\end{equation*}
with $q=\min(2,  \f{d(\gamma+1)}{(d-2)\gamma})$.
\end{proposition}

\begin{proof}
From Inequalities~\eqref{ine111}--\eqref{ine311}, we deduce that $\nabla n\in L^{2}(0,T,L^{2}(\Omega))$. Thus, using the Poincaré-Wirtinger inequality, we obtain that $n\in L^{2}(0,T,H^{1}(\Omega))$ uniformly in $\gamma,\sigma$. Next, we recall that in dimension $d$, $H^{1}\hookrightarrow L^{r}$ where $r=\frac{2d}{d-2}$ for $d>2$, $1\le r<\infty$ for $d=2$, and $1\le r\le\infty$ for $d=1$. Thus, $n^{1/2}\in L^{4}(0,T,L^{2r}(\Omega))$. Therefore, for all $\varphi\in L^{4}(0,T,W^{1,s}(\Omega))$, we can compute 
\begin{align*}
\Big|\int_{\Omega}\partial_{t}n\varphi\Big|&=\Big|\int_{\Omega}n^{1/2}n^{1/2}\nabla\mu\cdot\nabla\varphi+\int_{\Omega}n G(\mu)\varphi\Big|\\
&\le \|n^{1/2}\|_{L^{2r}(\Omega)}\|n^{1/2}\nabla\mu\|_{L^{2}(\Omega)}\|\nabla\varphi\|_{L^{s}}+C\|n\|_{L^{r}(\Omega)}\|\varphi\|_{L^{s/2}(\Omega)},\\
\end{align*}
because $\frac{2r}{r-1}=\frac{4d}{d+2}:=s$ for $d>2$. For $d\le 2$, we can consider any $r>1$. Assumption~\eqref{eq:assumption G} is used to control the source term in $L^{\infty}$. Integrating in time, using the triangle inequality on the integrals, Inequality~\eqref{ine511} and Hölder's inequality, we obtain the second estimate,
\begin{align*}
\Big|\int_{\Omega_{T}}\partial_{t}n\varphi\Big|\le C\|\varphi\|_{L^{4}(0,T,W^{1,s}(\Omega))}. 
\end{align*}
\\
The first Lebesgue estimate for $w$ is just an application of Inequality~\eqref{ine711}.
For the second Lebesgue estimate for $w$ we write Inequality~\eqref{ine611} as
\begin{equation*}
\frac{4\gamma}{(\gamma+1)^{2}}\int_{\Omega_{T}}|\nabla w^{\frac{\gamma+1}{2}}|^{2}\le C.	
\end{equation*}
\begin{sloppypar}
With Sobolev's embedding and the Poincaré-Wirtinger inequality, we obtain
\begin{equation*}
\frac{1}{\gamma}\|w^{\frac{\gamma+1}{2}}\|^{2}_{L^{2}(0,T,L^\frac{2d}{d-2}(\Omega))}\le C.     
\end{equation*}
This yields the second estimate. In addition, this estimate allows us to bound $w^{\gamma}$ uniformly in $L^\f{\gamma+1}\gamma (0,T,L^{\f{d(\gamma+1)}{(d-2)\gamma}}(\Omega))$.
Finally, we recall the last formula for $\mu$ in~\eqref{eq:defofw}, $\mu=w^{\gamma}-\Delta w$. Using the previous bound on $w^{\gamma}$ and the bound~\eqref{ine411}, we obtain the estimate on $\mu$.
\end{sloppypar}
This ends the proof of Proposition~\ref{compactnessh2}. 
\end{proof}

\subsection{Convergence \texorpdfstring{$\sg \to 0$}{}}\label{section:sig}

\commentout{To study the limit $\sigma\to 0$, the main difficulty is the strong convergence of $\mu$ and $\nabla w$ uniformly in $\sigma$. To overcome we need to consider a weaker type of solutions. We recall that $\mu=w^{\gamma}-\delta\Delta w$, and since we have uniform bounds in $\sigma$ for $\Delta w$ thanks to the entropy estimate, elliptic regularity gives us that $w$ is uniformly bounded in $L^{2}(0,T,H^{2}(\Omega))$. Furthermore, from Inequalities~\eqref{ine611}, \eqref{ine711} and the Poincaré-Wirtinger inequality, we have that $w^{\gamma}$ is bounded uniformly in $L^{\frac{\gamma+1}{\gamma}}(0,T,L^{\frac{d(\gamma+1)}{(d-2)\gamma}}(\Omega))\cap L^{\infty}(0,T,L^{\frac{\gamma+1}{\gamma}}(\Omega))$. Therefore, we are able to control $\mu$. At the end, and using similar arguments as in \cite{perthame_poulain}, we find 
}

To study the limit $\sigma\to 0$, we write the equivalent  RCH and GKS systems as
\begin{align*}
\partial_{t} n_\sigma-\Div(n_\sigma\nabla \mu_\sigma)=n_\sigma G(\mu_\sigma),\quad \text{in}\quad &(0,+\infty)\times\Omega,\\
\mu_\sigma =w^{\gamma}_\sigma-\delta\Delta w_\sigma,\quad w_\sigma=n_\sigma-\f\sg\delta\mu_\sigma\quad\text{in}\quad &(0,+\infty)\times\Omega .
\end{align*}

The main difficulty is the convergence of the term $n_{\sigma}\nabla\mu_{\sigma}$ in the first equation. We know that $n_{\sigma}\to n$ strongly in $L^{2}(\Omega_{T})$  thanks to the Lions-Aubin lemma and Proposition~\ref{compactnessh2}. In fact, the strong convergence is obtained in $L^{<\infty}(0,T,L^{2}(\Omega))$ (see below). No useful information can be obtained on $\nabla\mu_{\sigma}$  and for this reason,  after an integration by parts, it remains to prove the weak convergence in $L^{>1}(0,T,L^{2}(\Omega))$ of $\mu_{\sigma}$. The estimate on $\mu_{\sigma}$ given in Proposition~\ref{compactnessh2} is not enough in high dimension (i.e. $d>4$) and we need to consider weak solutions given by~\eqref{CH:weak1}-\eqref{CH:weak1_2} in the limit.  Strong convergence of $\mu_{\sigma}$ is also needed in order to identify the source term. Finally, with the definition of $w_{\sigma}$, notice that Inequalities~\eqref{ine111} show that $w_{\sigma}$ is a compact perturbation of $n_{\sigma}$ (see Proposition~\ref{strongconvw}) which has much higher regularity than $n_{\sigma}$ due to the entropy/energy structure of our system. Therefore the limit in the terms $\Div(n_{\sigma}\nabla\mu_{\sigma})$ and $n_{\sigma}G(\mu_{\sigma})$ is treated by writing $n_{\sigma}=w_{\sigma}+\text{'compact perturbation'}$. 

\begin{remark}
In the following proof and until the end,  we consider convolutions which we integrate on a bounded domain to prove compactness properties. This is just a formal writing because the functions we consider are only defined on $\Omega$. To make it rigorous, instead of considering a function $f$ we need to consider $f\zeta$ where $\zeta$ is a truncature function. Then, the compactness has to be obtained in domains $\Omega_{\eps}$ for every $\eps$ small enough where $\Omega_{\eps}\subset\Omega$ and every points in $\Omega_{\eps}$ is at distance $\eps$ of the boundary of $\Omega$. From now on, we ignore this technical issue in the proof below.    
\end{remark}

\begin{proposition}
\label{strongconvw}
With the assumptions of Theorem~\ref{convKSCH} we have $\nabla w_{\sigma},\nabla n_{\sigma}\to \nabla n$ strongly in $L^{2}(\Omega_{T})$ when $\sigma\to 0$.
\end{proposition}
\begin{proof}
Inequalities~\eqref{ine111} show that $\sigma\mu_{\sigma},$ $\sigma\nabla\mu_{\sigma}\to 0$ in $L^{2}(\Omega_{T})$. In addition, Proposition~\ref{compactnessh2} yields the weak convergence of $\nabla w_{\sigma},\nabla n_{\sigma}\rightharpoonup \nabla n$. It remains to prove strong convergence. We rely on arguments similar to those of Appendix~\ref{app:FK} using the Fréchet-Kolmogorov theorem. \\We know that $\nabla w_{\sigma}$ is bounded uniformly in $L^{\infty}(0,T,L^{2}(\Omega))\cap L^{2}(0,T,H^{1}(\Omega))$ using Inequalities~\eqref{ine311}-\eqref{ine411} and elliptic regularity. Then, by interpolation we have an $L^{p}(\Omega_{T})$ bound for some $p>2$. It remains to prove strong convergence in $L^{1}$.\\
We already have compactness in space thanks to the $H^{1}$ bound. With the notations of Appendix~\ref{app:FK}, it only remains to prove compactness in time, that means
 \begin{equation}
 \label{proofstrongconvergence}
 \int_{0}^{T-h}\int_{\Omega}|\nabla w_{\sigma}(t+h,\cdot)\ast\varphi_{\eps}(x)-\nabla w_{\sigma}(t,\cdot)\ast\varphi_{\eps}(x)|dxdt\le \theta(h),\quad\text{$\theta(h)\to 0$ when $h\to 0$},
 \end{equation} 
uniformly in $\sigma$.

From \eqref{eq:defofw} we can write

 \begin{multline*}
 \int_{0}^{T-h}\int_{\Omega}|\nabla w_{\sigma}(t+h,\cdot)\ast\varphi_{\eps}(x)-\nabla w_{\sigma}(t,\cdot)\ast\varphi_{\eps}(x)|dxdt\le \int_{0}^{T-h}\int_{\Omega}|(n_{\sigma}(t+h,\cdot)-n_{\sigma}(t,\cdot))\ast\nabla \varphi_{\eps}(x)|dxdt\\+\sigma\int_{0}^{T-h}\int_{\Omega}|(\nabla\mu_{\sigma}(t+h,\cdot)-\nabla\mu_{\sigma}(t,\cdot))\ast\varphi_{\eps}(x)|dxdt.
 \end{multline*}

The term $\sigma\nabla\mu_{\sigma}$ converges strongly to 0 in $L^{2}(\Omega_{T})$. This provides equicontinuity of the last term. The first term is dealt in the same way as in Appendix~\ref{app:FK}. This yields the strong convergence of $\nabla w_{\sigma}$.
Using once again the previous decomposition and  $\sigma\nabla\mu_{\sigma}\to 0$ show that $\nabla n_{\sigma}$ converges strongly.
\end{proof}

Now, we can prove Theorem~\ref{convKSCH}.

\begin{proof}
\textit{Step 1: Strong compactness for $w_{\sigma},n_{\sigma},\nabla w_{\sigma},\nabla n_{\sigma}$.} This step is a consequence of Proposition~\ref{compactnessh2} and Proposition~\ref{strongconvw}.

\textit{Step 2: Convergence of $n\nabla\mu$.}
Now, most of the convergences come from properties of the weak limit and the bounds provided at the beginning of this section. For the case of any dimension, we write $n_{\sigma}\nabla\mu_{\sigma}$ as 
\begin{equation*}
n_{\sigma}\nabla\mu_{\sigma}=w_{\sigma}\nabla\mu_{\sigma}+\frac{\sigma}{\delta}\mu_{\sigma}\nabla\mu_{\sigma}.
\end{equation*}
The last term converges strongly to 0 as $\sigma\to 0$. Indeed, from Inequalities~\eqref{ine111} together with the Sobolev injection and the Poincaré-Wirtinger inequality, we know that 
\begin{equation*}
\|\sigma^{1/2}\mu_{\sigma}\|_{L^{r}(\Omega_{T})}\le C,
\end{equation*}
 for some $r>2$ and $C$ independent of $\sigma$. Using also that $\mu$ is bounded in $L^{1}(\Omega_{T})$ and interpolation we find
 \begin{equation}
  \label{musigma}
  \|\sigma^{1/2}\mu\|_{L^{2}(\Omega_{T})}\le C\sigma^{\frac{1-\theta}{2}},\quad\theta=\frac{r}{2(r-1)}<1.  
 \end{equation}
 With Inequalities~\eqref{ine111}-\eqref{musigma} we find $\sigma\mu_{\sigma}\nabla\mu_{\sigma}=(\sigma^{1/2}\mu_{\sigma})(\sigma^{1/2}\nabla\mu_{\sigma})\to 0$ strongly.
We now treat the first term on the right hand side. We recall that $\mu_{\sigma}=w_{\sigma}^{\gamma}-\delta\Delta w_{\sigma}$ and therefore we write $w_{\sigma}\nabla\mu_{\sigma}$ as
\begin{equation}
\label{weakform3}
  w_{\sigma}\nabla\mu_{\sigma}=\frac{\gamma}{\gamma+1}\nabla w_{\sigma}^{\gamma+1}-\delta\Big[\nabla(w_{\sigma}\Delta w_{\sigma})+\nabla\frac{|\nabla w_{\sigma}|^{2}}{2}-\sum_{i=1}^{d}\partial_{i}(\partial_{i}w_{\sigma}\nabla w_{\sigma})\Big].
\end{equation}
We recall the strong convergence of $w_{\sigma}$ and $\nabla w_{\sigma}$ in $L^{2}(\Omega_{T})$ to $n,\nabla n$ and we have the weak convergence of $\Delta w_{\sigma}$ to $\Delta n$ thanks to Inequality~\eqref{ine411}. Finally we obtain the 
convergence in the distributional sense of the bracket in the right hand side. For the first term on the right hand side, we use Proposition~\ref{compactnessh2}. This provides that $w_{\sigma}^{\gamma+1}$ is bounded uniformly $L^{1}(0,T,L^{p}(\Omega))\cap L^{\infty}(0,T,L^{1}(\Omega))$ for some $p>1$ and therefore by interpolation in some $L^{q}(0,T,L^{r}(\Omega))$ for some $1<q<\infty$ and $r=\frac{1}{1-1/q+1/(qp)}$. Using the strong convergence of $w_\sigma$ and the Lebesgue dominated convergence theorem allows us to idenfity $w^{\gamma+1}$ in the limit. Using integration by parts we obtain the weak formulations~\eqref{CH:weak1}-\eqref{CH:weak1_2}.\\
The regularity of $n, \nabla n, \Delta n$ in the limit comes from the uniform bounds on $w_{\sigma}, \nabla w_{\sigma}, \Delta w_{\sigma}$ and the convergence of $\sigma\mu_{\sigma}$  to 0. \\
In dimension $d\le 4$ we have better regularity for $\mu$ and we can find another weak formulation. This is achieved in the following step.

\textit{Step 3: The weak formulation~\eqref{CH:weak3} in dimension $d\le 4$.}
In the case $d\le 4$ we have
\begin{equation*}
\|\mu_{\sigma}\|_{L^{\frac{\gamma+1}{\gamma}}(0,T,L^{2}(\Omega_{T}))}\le C,\quad \|\nabla w_{\sigma}\|_{L^{\infty}(0,T,L^{2}(\Omega))}\le C,
\end{equation*}
thanks to Proposition~\ref{compactnessh2} and Inequality~\eqref{ine311}. 
We write
\begin{equation*}
	\mu_{\sigma}\nabla n_{\sigma}=\mu_{\sigma}\nabla w_{\sigma}+\frac{\sigma}{\delta}\mu_{\sigma}\nabla\mu_{\sigma}.
\end{equation*}
We know from above that the second term of the right hand side strongly converges to 0. For the first term of the right hand side we know that $\nabla w_{\sigma}\to \nabla n$ strongly in $L^{2}(\Omega_{T})$. With the previous bound on $\nabla w_{\sigma}$ the convergence actually holds in every $L^{p}(0,T,L^{2}(\Omega))$ for $p<\infty$. Using the weak convergence of $\mu_{\sigma}$ we find the weak convergence of the product.\\
The weak convergence of $n_{\sigma}\mu_{\sigma}$ is similar. Therefore, with integration by parts of the term $\Div(n_{\sigma}\nabla\mu_{\sigma})$ against a test function $\chi$ we find the result.\\

\textit{Step 4: Identification of the source term when $d\le 4$.}
The last difficulty in the proof is to identify the source term (i.e $nG(\mu))$. Indeed, we do not know that $\mu_{\sigma}$ converges a.e. (up to a subsequence). However, we know that $n_{\sigma}$ converges a.e to $n$ and $G(\mu_{\sigma})$ is bounded. It remains to prove that $\mu_{\sigma}$ converges a.e. in the zone $\{n>0\}$.\\
For this reason we search for estimates on $n_{\sigma}\mu_{\sigma}$ to prove its convergence almost everywhere. Then, the convergence of $n_{\sigma}\mu_{\sigma}$ and $n_{\sigma}$ yields the convergence a.e. of $\mu_{\sigma}$ in the zone $\{n>0\}$. Now, we write $n_{\sigma}\mu_{\sigma}=w_{\sigma}\mu_{\sigma}+\frac{\sigma}{\delta}\mu_{\sigma}^{2}$ and compute
\[\nabla(n_{\sigma}\mu_{\sigma})=n_{\sigma}^{1/2}n_{\sigma}^{1/2}\nabla\mu_{\sigma}+\mu_{\sigma}\nabla w_{\sigma}+\frac{3\sigma}{2\delta}\nabla(\mu_{\sigma}^{2}).\]
As previously, the last term is bounded uniformly in $L^{1}(\Omega_{T})$ (and even converges to 0). The first term is also bounded in $L^{1}(\Omega_{T})$ with Proposition~\ref{compactnessh2} and Inequality~\eqref{ine511}.
 For the second term we use that $\mu_{\sigma}$ is bounded uniformly in $L^{\frac{\gamma+1}{\gamma}}(0,T,L^{\frac{d(\gamma+1)}{(d-2)\gamma}}(\Omega))$. In dimension $d\le 4$, Inequality~\eqref{ine311} provides a bound on the second term. Therefore we have compactness in space. 
The proof of the time compactness uses arguments and lemmas reported in Appendix~\ref{app:FK}. Since $\nabla(w\mu)$ is bounded, it only remains to show 
\begin{equation*}
\int_{0}^{T-h}\int_{\Omega}|w_{\sigma}(t+h,\cdot)\mu_{\sigma}(t+h,\cdot)\ast\varphi_{\eps}(x)-w_{\sigma}(t,\cdot)\mu_{\sigma}(t,\cdot)\ast\varphi_{\eps}(x)|dxdt\underset{h\to 0}{\longrightarrow} 0,\quad\text{uniformly in $\sigma$},
\end{equation*}
where the smooth functions $\varphi_{\eps}$ are defined in the Appendix~\ref{app:FK}.

We know using~\eqref{eq:defofw} that 
\begin{equation*}
w_{\sigma}\mu_{\sigma}=w_{\sigma}^{\gamma+1}-\delta w_{\sigma}\Delta w_{\sigma}=w_{\sigma}^{\gamma+1}-\frac{\delta}{2}\Delta w_{\sigma}^{2}+\delta|\nabla w_{\sigma}|^{2}.
\end{equation*}
Since $w_{\sigma}$ converges strongly in $L^{2}(\Omega_{T})$ and is bounded in $L^{\infty}(0,T,L^{\gamma+1}(\Omega))\cap L^{\gamma+1}\big((0,T),L^{\f{d(\gamma+1)}{d-2}}(\Omega)\big)$ thanks to Inequality~\eqref{ine711} and Proposition~\ref{compactnessh2}, we conclude by interpolation the strong convergence of $w_{\sigma}$ in $L^{\gamma+1}(\Omega_{T})$. Therefore the first term is equicontinuous by the converse of the Fréchet-Kolmogorov theorem.

It only remains to estimate
\begin{multline*}
 \int_{0}^{T-h}\int_{\Omega}|[w_{\sigma}(t+h,\cdot)\Delta w_{\sigma}(t+h,\cdot)-w_{\sigma}(t,\cdot)\Delta w_{\sigma}(t,\cdot)]\ast\varphi_{\eps}(x)|dxdt\\ \le 
 \int_{0}^{T-h}\int_{\Omega}\Big|\int_{\Omega}(|\nabla w_{\sigma}(t,y)|^{2}-|\nabla w_{\sigma}(t+h,y)|^{2})\varphi_{\eps}(x-y)\Big|dydxdt\\+ \int_{0}^{T-h}\int_{\Omega}\Big|\int_{\Omega}( w_{\sigma}^{2}(t+h,y)- w_{\sigma}^{2}(t,y))\Delta\varphi_{\eps}(x-y)\Big|dydxdt.
\end{multline*}

With the strong convergence of $\nabla w_{\sigma}$ and $w_{\sigma}^{2}$ we find the equicontinuity of the second and third term in the decomposition. This yields the equicontinuity with respect to time of $w\mu$.

Therefore thanks to the Fréchet-Kolmogorov theorem we have strong convergence in some $L^{p}$ and the convergence a.e. up to a subsequence of $\mu$ in the zone $n>0$. The Lebesgue dominated convergence theorem then allows us to identify the source term in the definition of the weak solutions. 
\end{proof}

%We recall here that since the GKS system~\eqref{eq:KS1}--\eqref{eq:KS2} and the RCH~\eqref{eq:CH1}--\eqref{eq:CH2} are perfectly equivalent, proving that solutions of the latter system converge to solutions of the DCH model~\eqref{eq:CH1-nonrelax}--\eqref{eq:CH2-nonrelax} in the limit $\sg\to 0$ shows that solutions of the GKS system converge as well. 

%This definition of weak solutions prevents us from studying limits of the system with $\sigma=0$ when $\gamma\to\infty$. For this reason, we study the incompressible limit for a fixed $\sigma>0$.

\section{Incompressible limit \texorpdfstring{$\gamma \to +\infty$}{}}
\label{sec:inpressible-limit}
We now fix $\delta, \sigma>0 $ and study the incompressible limit $\gamma\to\infty$ for the RCH-DKS system. These two constants are the main link between RCH and DKS models. We recall that in the case $\delta=0$, the incompressible limit is given in~\cite{Perthame-incompressible-visco}. In the case $\sigma=0$, the regularity provided by the DCH model is not sufficient to pass to the limit. From now on, we keep $\delta, \sigma>0$ fixed and we may consider them equal to 1 in some computations. We summarize the main bounds proved in this section in the following proposition

\begin{proposition}
\label{prop:summary}
\begin{comment}
\begin{align*}
n&\in C(0,T,L^{p}(\Omega))\cap L^{\infty}(\Omega_{T}),\quad\partial_{t}n\in L^{2}(0,T,H^{-1}(\Omega)),\\
w&\in C(0,T,L^{p}(\Omega))\cap L^{\infty}(0,T,H^{3}(\Omega))\cap L^{\infty}(\Omega_{T}),\quad \Delta w\in L^{\infty}(\Omega_{T}),\quad\partial_{t}w\in L^{2}(0,T,H^{1}(\Omega)),\\
w^{\gamma}&\in C(0,T,L^{p}(\Omega))\cap L^{\infty}(0,T,H^{1}(\Omega))\cap L^{\infty}(\Omega_{T}),\quad\partial_{t}w^{\gamma}\in L^{2}(0,T,H^{-1}(\Omega)), 
\end{align*}
\end{comment}
For all $T>0$, there exists a constant $C(T)$ independent of $\gamma$ such that the weak solutions built in Section~\ref{sec:existence} satisfy

\begin{align}
&\|\mu\|_{L^{2}(0,T,H^{1}(\Omega))}\le C, \quad \|n\|_{L^{2}(0,T,H^{1}(\Omega))}\le C,\label{ine31}\\
&\|n\|_{L^{\infty}(\Omega_{T})}\le C,\label{ine32}\\
&\begin{array}{ll}\|w\|_{L^{\infty}(0,T,H^{2}(\Omega))}\le C,\quad \|w^{\gamma}\|_{L^{\infty}(\Omega_{T})}\le C, \quad \|\nabla w^{\gamma}\|_{L^{\infty}(0,T,L^{1}(\Omega))}\le C,\\
\|\Delta w\|_{L^{\infty}(\Omega_{T})}\le C,\end{array}\label{ine33}\\
&\begin{array}{ll}\|\partial_{t}n\|_{L^{2}(0,T,H^{-  1}(\Omega))}\le C,\quad \|\partial_{t}w\|_{L^{2}(0,T,H^{1}(\Omega))}\le C,\quad  \|\partial_{t}\mu\|_{L^{2}(0,T,H^{-1}(\Omega))}\le C, \\
\|\partial_{t}w^{\gamma}\|_{L^{2}(0,T,H^{-1}(\Omega))}\le C.\end{array}\label{ine34}
\end{align}
The weak solutions in the case of no source term, i.e $G=0$, are unique.
\end{proposition}

\subsection{Uniform a priori estimates in Proposition~\ref{prop:summary}}
\label{sec:apriori-gamma}

\begin{proof}[Proof of Proposition~\ref{prop:summary}]
We start with the first two estimates~\eqref{ine31} of Proposition~\ref{prop:summary}

The first inequality is a consequence of Inequalities~\eqref{ine111}. The second inequality has been proven in the previous section. 

\textit{$L^{\infty}$ bound for n.} We now establish Inequality~\eqref{ine32} under the assumption~\eqref{initcond} on the initial condition. The proof requires a variant of Gagliardo-Nirenberg inequality, namely there exists $C>0$ such that for every $0<\eps< 1/2$ and every $v\in H^{1}(\Omega)$,
\begin{equation}
\|v\|_{L^{2}(\Omega)}^{2}\le \eps \|\nabla v\|^{2}_{L^{2}(\Omega)}+\frac{C}{\eps^{d/2}}\|v\|_{L^{1}(\Omega)}^{2}\label{eq:GN}.
\end{equation}

This inequality is an application of the classical Gagliardo-Nirenberg and Young inequalities. We refer the reader to \cite{Cholewa}, Equation~(9.3.8).

First, to begin the proof, we wish to stress a few comments. The proof of this proposition relies on the use of the Alikakos iteration method~\cite{Alikakos}. We prove it for smooth solutions of the equation but the method can be applied for weak solutions, since the bounds only depend on the a priori estimates already found.

We start by choosing $\sigma=\delta=1$ to simplify the notation. We notice however that the $L^{\infty}$ bound is not uniform in $\sigma$ and therefore this result does not apply to Subsection~\ref{section:sig}. In fact, with the same proof, it is possible to show that the $L^{\infty}$ bound varies as $1/\sigma$. We multiply Equation~\eqref{eq:KS1} by $n^{2^{k}-1}$, integrate over the domain, and after integration by parts that uses the boundary conditions~\eqref{boundary}, we obtain 
\begin{equation}
\frac{1}{2^{k}}\frac{d}{dt}\int_{\Omega}n^{2^{k}}+(2^{k}-1)\int_{\Omega}n^{2^{k}-1}|\nabla n|^{2}=(2^{k}-1)\int_{\Omega}n^{2^{k}-1}\nabla w\cdot\nabla n+\int_{\Omega}n^{2^{k}}G(\mu).
\label{eq:alikakos-1}
\end{equation}

Then, multiplying Equation~\eqref{eq:KS2} by $n^{2^k}$, integrating over $\Omega$, and after integration by parts, we have from the non-negativity of both $n$ and $w$ (see Proposition~\ref{nonnegativities})
\begin{equation}
   \int_\Omega n^{2^k-1}\nabla w \cdot \nabla n=\frac{1}{2^{k}}\int_{\Omega}-\Delta w n^{2^{k}} \le \f1{2^{k}} \int_\Omega n^{2^k+1}.
   \label{eq:Alikakos-2}
\end{equation}
Therefore, rearranging the second term on the left-hand side of Equation~\eqref{eq:alikakos-1}, passing it to the right-hand side, using Inequality~\eqref{eq:Alikakos-2}, and, finally, using $n^{2^{k}}\le n^{2^{k}+1}+1$ for the second term of the right hand side, we obtain
\begin{equation*}
\frac{1}{2^{k}}\frac{d}{dt}\int_{\Omega}n^{2^{k}}\le -\frac{4(2^{k}-1)}{(2^{k}+1)^{2}}\int_{\Omega}|\nabla n^{\frac{2^{k}+1}{2}}|^{2}+\Big(\frac{2^{k}-1}{2^{k}}+C\Big)\int_{\Omega}n^{2^{k}+1}+C.
\end{equation*}
This means exactly that for some $C>0$,
\begin{equation*}
\frac{d}{dt}\|n\|_{L^{2^{k}}(\Omega)}^{2^{k}}\le -\frac{4(2^{k}-1)2^{k}}{(2^{k}+1)^{2}}\|\nabla n^{\frac{2^{k}+1}{2}}\|_{L^{2}(\Omega)}^{2}+C2^{k}\|n\|_{L^{2^{k}+1}}^{2^{k}+1}+C2^{k}.
\end{equation*}
Applying Equation~\eqref{eq:GN} with $v=n^{\frac{2^{k}+1}{2}}$, we obtain for any $\eps>0$

\begin{equation*}
\frac{d}{dt}\|n\|_{L^{2^{k}}(\Omega)}^{2^{k}}\le \Big(C2^{k}-\frac{4(2^{k}-1)2^{k}}{\eps(2^{k}+1)^{2}}\Big)\|n\|_{L^{2^{k}+1}}^{2^{k}+1}+\frac{C}{\eps^{(d+2)/2}}\frac{4(2^{k}-1)2^{k}}{(2^{k}+1)^{2}}\|n^{\frac{2^{k}+1}{2}}\|_{L^{1}(\Omega)}^{2}+C2^{k}.
\end{equation*}
Choosing $\eps=C2^{-k}$ in order to let the first term of the right-hand side to be non-positive, leads to 
\begin{equation*}
\frac{d}{dt}\|n\|_{L^{2^{k}}(\Omega)}^{2^{k}}\le C2^{(d+2)k/2}\|n^{\frac{2^{k}+1}{2}}\|_{L^{1}(\Omega)}^{2}+C2^{k}.
\end{equation*}
Moreover, using Riesz-Thorin interpolation theorem, we have
\begin{equation*}
\frac{d}{dt}\|n\|_{L^{2^{k}}(\Omega)}^{2^{k}}\le C2^{(d+2)k/2}\|n\|_{L^{2^{k}}(\Omega)}^{2}\|n\|_{L^{2^{k-1}}(\Omega)}^{2^{k}-1}+C2^{k}.
\end{equation*}
Denoting $m_{k}=\sup_{t\in(0,T)}\|n(t)\|_{L^{2^{k}}(\Omega)}$ and after integrating in time, we obtain
\begin{equation*}
m_{k}\le \Big(C2^{\frac{(d+2)k}{2}}m_{k-1}^{2^{k}-1}m_{k}^{2}+C2^{k}\Big)^{1/2^{k}}.
\end{equation*}
Following~\cite{Cholewa} (Lemma~9.3.1, p.213) the sequence $m_{k}$ can be dominated by $m_{k}'$  which satisfies
\begin{equation*}
    m_{k}'=(C2^{\frac{(d+2)k}{2}})^{1/2^{k}}m_{k-1}'^{1-1/2^{k}}m_{k}'^{1/2^{k-1}}\quad\text{with $C$ large enough and $m_{0}'\ge 1.$},
\end{equation*}
i.e,
\begin{equation*}
    m_{k}'=(C2^{\frac{(d+2)k}{2}})^{\frac{1}{2^{k}-2}}m_{k-1}'^{\frac{2^{k}-1}{2^{k}-2 }}.
\end{equation*}
Letting $k\to\infty$ and by induction we find $\|n\|_{L^{\infty}(\Omega_{T})}=m_{\infty}\le m_{\infty}'   \le C$. We refer to~\cite{Cholewa} for the details. This yields the result. This in turn provides also higher regularity for $w$ thanks to Equation~\eqref{eq:KS2} and one can find estimates~\eqref{ine33}.

\textit{Proof of~\eqref{ine33}.} The first estimate is a consequence of Equation~\eqref{eq:KS2} together with elliptic regularity.\par
For the second estimate we multiply Equation~\eqref{eq:KS2} by $w^{\gamma(r-1)}$. After an integration by parts and using the nonnegativity of $n,w$ we find for every $t$,
\begin{equation*}
\|w^{\gamma}(t)\|_{L^{r}(\Omega)}^{r}\le\|n(t) w(t)^{\gamma(r-1)}\|_{L^{r}(\Omega)}.
\end{equation*}
With Hölder inequality, we obtain 
\begin{equation*}
\|w^{\gamma}(t)\|_{L^{r}(\Omega)}\le\|n(t)\|_{L^{r}(\Omega)}.  
\end{equation*}
Letting $r\to\infty$ thanks to Inequality~\eqref{ine32} and taking the supremum in time yields the result. We refer the reader to Theorem 2.14 in \cite{adams2003sobolev} for further details.
The third inequality is just the result of the two previous inequalities as well as Equation~\eqref{eq:KS2}. To get the fourth inequality we differentiate in space Equation~\eqref{eq:KS2} and multiply by $\sgn(\partial_{x}w^{\gamma})$. Since $w$ is nonnegative, $\sgn(\partial_{x}w^{\gamma})= \sgn(\partial_{x}w)$. Using integration by parts on the first term yields the result. The computations can be made rigorous with the derivative of a convex approximation of the absolute value.

\begin{remark}
With Inequality~\eqref{ine32} and Equation~\eqref{eq:KS2} it is possible to find that $w\in L^{2}(0,T,H^3(\Omega))$ and therefore $w^{\gamma}\in L^{2}(0,T,H^{1}(\Omega))$ thanks to Equation~\eqref{eq:KS2}. However, the bound is not uniform in $\gamma$ and we cannot gain compactness. 
\end{remark}

\begin{remark}
When $G=0$, Inequalities~\eqref{ine33} provides uniqueness of the weak solutions, see Appendix~\ref{app:uniqueness}.
\end{remark}

\textit{Time compactness of $n$.} Compactness in time for $n$ follows using once again the Lions-Aubin lemma with the $(L^{4}(0,T,W^{1,s}(\Omega)))'$ bound on $\partial_{t}n$ from proposition~\ref{compactnessh2}. One can prove compactness in time for $w$ using the Fréchet-Kolmogorov theorem. However, with the previous regularity results, we can get the better bounds~\eqref{ine34} on the time derivative of $n$, $w$, $p=w^{\gamma}$ and $\mu$. 

\textit{Proof of~\eqref{ine34}.} Since we have found that $n\in L^{\infty}(\Omega_{T})$, we have for any test function $\phi \in L^{2}(0,T;H^1(\Omega))$, 
\[
    \abs{\int_{\Omega_T} \p_t n \phi  } \le \norm{n}_{L^\infty(\Omega_T)}^{1/2} \int_{\Omega_T} \abs{ n^{1/2} \nabla \mu \nabla \phi}\le C \norm{\nabla \phi}_{L^2(\Omega_T)},
\]
where we have used~\eqref{ine32} and~\eqref{ine511}. Hence, we have ${\partial_{t}n\in L^{2}(0,T,H^{-1}(\Omega))}$ uniformly. To find compactness of $\partial_{t}\mu$ it is enough to find compactness for $\partial_{t}w$ thanks to Equation~\eqref{eq:defofw}. 
Computing the time derivative of Equation~\eqref{eq:KS2}, multiplying by a test function $\phi \in L^{2}(0,T;H^1(\Omega))$, using the notation $z=\p_t w$, and integrating in space and time, we have 
\[
\int_{\Omega_T} \sigma \nabla z \cdot \nabla \phi + \f{\sigma}{\delta}\gamma w^{\gamma-1}z \phi +z\phi= \int_{\Omega_T} \partial_{t}n \phi.	
\]
Choosing $\phi = z$, and from Young's Inequality, we have
\[
    \int_{\Omega_T} \left(\sigma-\f1{2\kappa}\right) \abs{\nabla z}^2  + \left(\f{\sigma}{\delta}\gamma w^{\gamma-1}+1-\f1{2\kappa}\right)\abs{z}^2 \le 2 \kappa  \norm{\partial_{t}n}^2_{L^2\left(0,T;H^{-1}(\Omega)\right)}.	
\]
Therefore, choosing $\kappa$ large enough, we obtain that $\p_t w \in L^{2}(0,T,H^{1}(\Omega))$. Using Equation~\eqref{eq:KS2} provides the compactness for $\partial_{t}w^{\gamma}$. This achieves the proof of~\eqref{ine34} and Proposition~\ref{prop:summary}.
\end{proof}

Finally, one can show that the regularity of the solutions provides continuity with respect to time

\begin{proposition}
Assume~\eqref{initcond}--\eqref{eq:assumption G}, the functions $n$, $w$ and $w^{\gamma}$ belong to $C(0,T,L^{p}(\Omega))$ for every $1\le p<~\infty$.
\end{proposition}

\subsection{Convergence \texorpdfstring{$\gamma \to +\infty$}{}}

With all the ingredients of the previous subsections we find
\begin{thm}[Incompressible limit] \label{th:convergence-gamma}
Assume~\eqref{initcond}--\eqref{eq:assumption G} and let $(n_{\sigma,\gamma},\mu_{\sigma,\gamma})$ be a weak solution to the RDCH model~\eqref{eq:CH1}--\eqref{eq:CH2}. Then, when $\gamma\to\infty$, after extraction of subsequences, $(n_{\sigma,\gamma},\mu_{\sigma,\gamma},w_{\sigma,\gamma})\to (n_{\sigma,\infty},\mu_{\sigma,\infty},w_{\sigma,\infty})$ with the regularity $n_{\sigma,\infty}\in L^{2}(0,T,H^{1}(\Omega))\cap L^{\infty}(\Omega_{T})$, $\partial_{t}n_{\sigma,\infty}\in L^{2}(0,T,H^{-1}(\Omega))$,  $w_{\sigma,\infty}\in C(0,T,L^{p}(\Omega))\cap L^{\infty}(0,T,W^{2,p}(\Omega))$ and $\mu_{\sigma,\infty}\in C(0,T,L^{p}(\Omega))\cap L^{2}(0,T,H^{1}(\Omega))$ for every $1\le p<\infty$. These functions  satisfy in the weak sense
\begin{equation*}
\partial_{t}n_{\sigma,\infty}-\Div(n_{\sigma,\infty}\nabla\mu_{\sigma,\infty})=n_{\sigma,\infty}G(\mu_{\infty}).
\end{equation*}
and
\begin{equation*}
\mu_{\sigma,\infty}=p_{\sigma,\infty}-\delta\Delta w_{\sigma,\infty},\quad -\sigma\Delta w_{\sigma,\infty}+\frac{\sigma}{\delta}p_{\sigma,\infty}+w_{\sigma,\infty}=n_{\sigma,\infty},\quad w_{\sigma,\infty}=n_{\sigma,\infty}-\frac{\sigma}{\delta}\mu_{\sigma,\infty}.
\end{equation*}
where $p_{\sigma,\infty}$ is the strong~$L^{p}(\Omega_{T})$-limit of $w^{\gamma}$ and belongs to $L^{\infty}(0,T,L^{\infty}(\Omega))\cap  C(0,T,L^{p}(\Omega))$. \commentout{Moreover $\nabla p\in L^{\infty}(0,T,\mathcal{M}(\Omega))$ where $\mathcal{M}(\Omega)$ denotes the space of Radon measures.}
\end{thm}
\begin{proof}
For the first term on the left hand side, the weak convergence of $\partial_{t}n_{\sigma,\gamma}$ given by~\eqref{ine34} is enough. For the second term of the left hand side, we use~\eqref{ine31}--\eqref{ine32} to prove the weak convergence of $n_{\sigma,\gamma}\nabla\mu_{\sigma,\gamma}$. Then, to identify its limit, we use the strong convergence of $n_{\sigma,\gamma}$ from~\eqref{ine31}-\eqref{ine34} and the weak convergence of $\nabla\mu_{\sigma,\gamma}$ given by~\eqref{ine31}. For the term on the right hand side, we use the strong convergence of $n_{\sigma,\gamma}$ and $\mu_{\sigma,\gamma}$. Finally for the equation on $\mu_{\sigma,\gamma}$ and $w_{\sigma,\gamma}$ weak convergence is enough. To prove strong convergence of $p_{\sigma,\gamma}=w^{\gamma}_{\sigma,\gamma}$ we use the bounds on $\partial_{t}p_{\sigma,\gamma}, \nabla p_{\sigma,\gamma}$ from~\eqref{ine33}--\eqref{ine34}. An application of the Fréchet-Kolmogorov theorem as in Appendix~\ref{app:FK} yields the strong convergence. The continuity with respect to time follows from the regularity of the solutions. 
\end{proof}

Moreover, we find two propositions on $w_{\sigma,\infty}$, the first one provides an $L^{\infty}$-bound on $w_{\sigma,\infty}$, and the second one gives some information about the behaviour of the potential in the zones where $w_{\sigma,\infty}\ne 1$.

\begin{proposition}[$L^\infty$-bound for $w_{\sigma,\infty}$]
\label{winf<1}
For the limit solution $w_{\sigma,\infty}$ defined in Theorem~\ref{th:convergence-gamma}, we have, with $w_{\sigma,\gamma}=n_{\sigma,\gamma}-\frac{\sigma}{\delta}\mu_{\sigma,\gamma}$,
\begin{equation*}
\|w_{\sigma,\infty}\|_{L^{\infty}(\Omega_{T})}\le 1,	
\end{equation*}

\end{proposition}

\begin{proof}
In the case $\sigma>0$, this estimate is trivial with the $L^{\infty}$ bound on $w^{\gamma}$ from Proposition~\ref{prop:summary}. However, we also provide a proof that works in the case $\sigma=0$, i.e., when $w_{0,\gamma}=n_{0,\gamma})$. 

We start by using Inequality~\eqref{ine711}, we have 
\begin{equation*}
\|w_{\sigma,\gamma}(t)\|_{L^{\gamma+1}(\Omega)}\le (C(\gamma+1))^{1/(\gamma+1)}\le C^{1/(\gamma+1)}.	
\end{equation*}
By interpolation, and with $\frac{1}{q}=\theta+\frac{1-\theta}{\gamma+1}$ where $q\in(1,\gamma+1)$, we have
\begin{equation*}
\|w_{\sigma,\gamma}(t)\|_{L^{q}(\Omega)}\le \|w_{\sigma,\gamma}(t)\|_{L^{1}(\Omega)}^{\theta}\|w_{\sigma,\gamma}(t)\|_{L^{\gamma+1}}^{1-\theta}.	
\end{equation*}
From the Cauchy-Schwarz inequality for the $L^{1}$ norm and from the previous inequality for the $L^{\gamma+1}$-norm, we easily find two constants $C,\widetilde{C}$ such that
\begin{equation*}
\|w_{\sigma,\gamma}(t)\|_{L^{q}(\Omega)}\le \widetilde{C}^{\theta}C^{(1-\theta)/(\gamma+1)}.	
\end{equation*}
Since we know that $w_{\gamma}\rightharpoonup w_{\infty}$, and by lower semi-continuity of the norm as well as the fact that $\theta\to 1/q$ when $\gamma\to\infty$,  we have
 \begin{equation*}
 \|w_{\sigma,\infty}\|_{L^{\infty}(0,T,L^{q}(\Omega))}\le\liminf_{\gamma\to\infty} \|w_{\sigma,\gamma}\|_{L^{\infty}(0,T,L^{q}(\Omega))}\le \widetilde{C}^{1/q},
 \end{equation*}
 for any $q\in(1,\infty)$.
 Therefore, letting $q\to\infty$, we obtain
  \begin{equation*}
 \|w_{\sigma,\infty}\|_{L^{\infty}(\Omega_{T})}\le\liminf_{q\to\infty} \|w_{\sigma,\infty}\|_{L^{\infty}(0,T,L^{q})(\Omega)}\le 1.
 \end{equation*}
\end{proof}

Compared to previous results on incompressible limits for living tissue models (see \eg \cite{Perthame-Hele-Shaw}), we have a slightly different relation linking the pressure and the density. We have the following proposition.
\begin{proposition}[Relation between $p_\infty$ and $w_\infty$]
The following relation holds for the limits of $p_{\sigma,\gamma}=w^{\gamma}_{\sigma,\gamma}$ and $w_{\sigma,\gamma}$,
\begin{equation*}
p_{\sigma,\infty}(w_{\sigma,\infty}-1)=0,\quad \text{a.e. in}\quad \Omega_T.	
\end{equation*}
\end{proposition}

\begin{proof}
The inequality $p_{\sigma,\infty}(w_{\sigma,\infty}-1)\le 0$ is found in a straightforward manner using Proposition~\ref{winf<1}, and the fact that $w_{\sigma,\gamma}\ge 0 \implies w_{\sigma,\gamma}^{\gamma}\ge 0\implies p_{\sigma,\infty}\ge 0$.

It remains to show that $p_{\sigma,\infty}(w_{\sigma,\infty}-1)\ge 0$. We borrow the argument of \cite{Lions-Masmoudi}. For $\nu>0$, there exists $\gamma_{0}$ such that for $\gamma\ge\gamma_{0}$,

\begin{equation*}
    w^{\gamma+1}_{\sigma,\gamma}\ge w^{\gamma}_{\sigma,\gamma}-\nu.
\end{equation*}
Then, from the convergence of $w_{\sigma,\gamma}$ and $w^{\gamma}_{\sigma,\gamma}$ we know that $w_{\sigma,\gamma}^{\gamma}w_{\sigma,\gamma}$ converges  to $p_{\sigma,\infty}w_{\sigma,\infty}$. Passing to the limit we get \begin{equation*}
p_{\sigma,\infty}w_{\sigma,\infty}\ge p_{\sigma,\infty}-\nu,
\end{equation*}
for every $\nu>0$. Letting $\nu\to 0$ yields the result.

\end{proof}

\begin{remark}
From this result we find that in the zone $w_{\infty}\ne 1,$ we get $p_{\infty}=0,$  and thus $\mu_{\infty}=-\delta\Delta w_{\infty}$ which can be interpreted as a term representing surface tension.
\end{remark}

When the relaxation parameter satisfy $\sigma=0$, we expect to have $w=n$. Therefore, in the zone $\Omega_{+}(t)=\{x, p_{\infty}(t,x)>0\}$ we obtain that the density stays constant $n_{\infty}=1$. With Equation~\eqref{eq:CH1} this means that formally 
\begin{equation*}
\begin{cases}
-\Delta\mu=G(\mu)\quad\text{in $\Omega_{+}$}, \\
\mu=-\delta\Delta n\quad\text{on $\partial\Omega_{+}$}.
\end{cases}
\end{equation*}

\section{Existence of weak solutions}\label{section:eps}
\label{sec:existence}
The proof of existence of weak solutions for system~\eqref{eq:CH1}--\eqref{eq:CH2} follows the standard method for the DCH (see \eg \cite{Garcke-CH-deg,perthame_poulain,Dai-Du}). We start by regularizing the model: using a small positive parameter $\eps$ we define a positive approximation of the degenerate mobility $b(n) := n$. Existence of a solution to the regularized system is found using standard methods for nonlinear parabolic equations. 
Then, we derive a priori estimates on the regularized problem that allow us to pass to the limit $\eps\to 0$ and, hence, show the existence of weak solution for System~\eqref{eq:CH1}--\eqref{eq:CH2}.

\subsection{Regularized mobility}
We consider a small parameter $0<\eps<1$ and define the regularized mobility
\begin{equation}
\label{regmob}
B_{\eps}(n)=
\begin{cases}
\frac{1}{\eps}\quad\text{for $n\ge \frac{1}{\eps}$,}	\\
\eps\quad\text{for $n\le\eps$,}\\
n\quad\text{otherwise}.
\end{cases}
\end{equation}
Thereby, we write the regularized analogous of System~\eqref{eq:CH1}--\eqref{eq:CH2} 
\begin{align}
\partial_{t} n_\eps-\Div(B_{\eps}(n_\eps)\nabla \mu_\eps)=n_\eps G(\mu_\eps),\quad \text{in}\quad &(0,+\infty)\times\Omega,\label{CHR1}\\
-\sigma\Delta\mu_\eps +\mu_\eps =(n_\eps-\frac{\sigma}{\delta}\mu_\eps)^{\gamma}-\delta\Delta n_\eps,\quad \text{in}\quad &(0,+\infty)\times\Omega,\label{CHR2}
\end{align}
supplemented with the zero-flux boundary conditions 
\begin{equation}
\label{boundaryR}
\frac{\partial w_\eps}{\partial\nu}=n_\eps\frac{\partial\mu_\eps}{\partial\nu}=0	\quad\text{on}\quad (0,\infty)\times\partial\Omega,
\end{equation}
where $w_{\eps}$ is defined by~\eqref{eq:defofw} and with the initial conditions~\eqref{initcond}.

We have the following existence theorem
\begin{thm}[Weak solutions for the regularized system] \label{th:global-weak-sol}
There exists a pair of functions $(n_{\eps},\mu_{\eps})$ such that for all $T\ge 0$, 
\begin{align*}
n_{\eps}&\in L^{2}(0,T,H^{1}(\Omega)),\quad\partial_{t}n_{\eps}\in L^{2}(0,T,H^{-1}(\Omega)),\\
\mu_{\eps}&\in L^{2}(0,T,H^{1}(\Omega)),\quad\partial_{t}\mu_{\eps}\in L^{2}(0,T,H^{-1}(\Omega)),\\
w_{\eps}&\in L^{2}(0,T,H^{2}(\Omega))\cap L^{\infty}(0,T,H^{1}(\Omega)),\quad w_{\eps}^{\gamma}\in L^{2}(0,T,L^{2}(\Omega))\quad\partial_{t}w_{\eps}\in L^{2}(0,T,H^{1}(\Omega)).
\end{align*}
These functions satisfy the regularized Cahn-Hilliard model~\eqref{CHR1}--\eqref{boundaryR} in the following weak sense: for all test functions $\chi\in L^{2}(0,T,H^{1}(\Omega)),$ it holds
\begin{align}
\int_{0}^{T}\langle \chi,\partial_{t}n_{\eps}\rangle&=-\int_{\Omega_{T}}B_{\eps}(n_{\eps})\nabla\mu_{\eps}\nabla\chi+\int_{\Omega_{T}}n_{\eps}G(\mu_{\eps})\chi,\\	
\sigma\int_{\Omega_{T}}\nabla\mu_{\eps}\nabla\chi+\int_{\Omega_{T}}\mu_{\eps}\chi&=\delta\int_{\Omega_{T}}\nabla n_{\eps}\nabla\chi+\int_{\Omega_{T}}\Big(n_{\eps}-\frac{\sigma}{\delta}\mu_{\eps}\Big)^{\gamma}\chi,\\
 -\sg  \Delta w_{\eps} + \f\sg\delta w_{\eps}^{\gamma}  + w_{\eps}  &= n_{\eps} \quad\text{a.e.},\label{eq:forw}\\
 \mu_{\eps}&=w_{\eps}^{\gamma}-\delta\Delta w_{\eps}\quad\text{a.e}.
\end{align}

\end{thm}

\begin{proof}
\textit{Step 1. Galerkin approximation.}
We consider $\{\phi_{i}\}_{i\in\N}$. the eigenfunctions of the Laplace operator with zero Neumann boundary conditions.
\begin{equation}
-\Delta\phi_{i}=\lambda_{i}\phi_{i}\in\Omega\quad\text{with}\quad \nabla\phi_{i}\cdot\nu=0\quad\text{on $\partial\Omega$},	
\label{eq:base-galerkin}
\end{equation}
which form an orthogonal basis of both $H^{1}(\Omega)$ and $L^{2}(\Omega)$ and we normalize them such that $(\phi_{i},\phi_{j})_{L^{2}(\Omega)}=\delta_{ij}.$ Furthermore we assume without loss of generality that $\lambda_{1}=0.$\\
We consider the following discrete approximation of System~\eqref{CHR1}-\eqref{CHR2}
\begin{align}
n^{N}(t,x)&=\sum_{i=1}^{N}c_{i}^{N}(t)\phi_{i}(x),\quad \mu^{N}(t,x)=\sum_{i=1}^{N}d_{i}^{N}(t)\phi_{i}(x),\\
\int_{\Omega}\partial_{t}n^{N}\phi_{j}&=-\int_{\Omega}B_{\eps}(n^{N})\nabla\mu^{N}\nabla\phi_{j}+\int_{\Omega}n^{N}G(\mu^{N})\phi_{j}, \quad \text{for $j=1,...,N,$}\label{Galerkin1}\\
\int_{\Omega}\mu^{N}\phi_{j}&=\delta\int_{\Omega}\nabla\Big(n^{N}-\frac{\sigma}{\delta}\mu^{N}\Big)\nabla\phi_{j}+\int_{\Omega}\Big(n^{N}-\frac{\sigma}{\delta}\mu^{N}\Big)^{\gamma}\phi_{j}, \quad \text{for $j=1,...,N,$}\label{Galerkin2}\\
n^{N}(0,x)&=\sum_{i=1}^{N}(n_{0},\phi_{i})_{L^{2}(\Omega)}\phi_{i}.
\end{align}
where the coefficients $c_{j}^{N}, d_{j}^{N}$ for $j=1,..,N$ are determined by 
\begin{align}
\partial_{t}c_{j}^{N}&=-	\int_{\Omega}B_{\eps}\Big(\sum_{i=1}^{N}c_{i}^{N}\phi_{i}\Big)\nabla\mu^{N}\nabla\phi_{j}+c_{j}^{N}\int_{\Omega}G(\mu^{N})\label{initvalue1},\\
d_{j}^{N}(1+\sigma\lambda_{j})&=\delta\lambda_{j}c_{j}^{N}+\int_{\Omega}\Big(\sum_{i=1}^{N}(c_{i}^{N}-\frac{\sigma}{\delta}d_{i}^{N})\phi_{i}\Big)^{\gamma}\phi_{j}\label{initvalue2},\\
c_{j}^{N}(0)&=(n_{0},\phi_{j})_{L^{2}(\Omega)}\label{initvalue3}.
\end{align}

Since the right hand side of Equation~\eqref{initvalue1} depends continuously on the coefficients $c_{j}^{N}$, standard results on ODE systems gives the existence and uniqueness of a local solution to the initial value problem~\eqref{initvalue1}--\eqref{initvalue3}.\par

\noindent \textit{Step 2. Inequalities and convergences.} Multiplying Equation~\eqref{Galerkin1} by $d_{j}(t)$ and summing over $j$ leads to 
\begin{equation*}
\int_{\Omega}\partial_{t}n^{N}\mu^{N}=-\int_{\Omega}	B_{\eps}(n^{N})|\nabla\mu^{N}|^{2}+\int_{\Omega}n^{N}G(\mu^{N})\mu^{N}.
\end{equation*}
Rearranging the left-hand side, we obtain 
\begin{equation*}
\int_{\Omega}\partial_{t}n^{N}\mu^{N}=\int_{\Omega}\partial_{t}\Big(n^{N}-\frac{\sigma}{\delta}\mu^{N}\Big)\mu^{N}+\frac{1}{2}\frac{\sigma}{\delta}\frac{d}{dt}\int_{\Omega}|\mu^{N}|^{2}.	
\end{equation*}
Using in Equation~\eqref{Galerkin2}, $\phi_j = \f{d}{dt}(c_{j}^{N}-\frac{\sigma}{\delta}d_{j}^{N})\phi_{j}$ and summing over $j$ , we have 
\begin{equation*}
\int_{\Omega}\partial_{t}\Big(n^{N}-\frac{\sigma}{\delta}\mu^{N}\Big)\mu^{N}=\frac{\delta}{2}\frac{d}{dt}\int_{\Omega}\Big|\nabla\Big(n^{N}-\frac{\sigma}{\delta}\mu^{N}\Big)\Big|^{2}+\frac{d}{dt}\int_{\Omega}\frac{\Big(n^{N}-\frac{\sigma}{\delta}\mu^{N}\Big)^{\gamma+1}}{\gamma+1}.
\end{equation*}
Altogether, we obtain the discrete energy dissipation
\begin{equation}
\label{Energy-Galerkin}
\frac{d}{dt}E(t)+\int_{\Omega}B_{\eps}(n^{N})|\nabla\mu^{N}|^{2}=\int_{\Omega}n^{N}G(\mu^{N})\mu^{N},	
\end{equation}
in which the energy is defined by
\begin{equation*}
E(t)=\frac{\delta}{2}\int_{\Omega}\Big|\nabla\Big(n^{N}-\frac{\sigma}{\delta}\mu^{N}\Big)\Big|^{2}+\int_{\Omega}\frac{\Big(n^{N}-\frac{\sigma}{\delta}\mu^{N}\Big)^{\gamma+1}}{\gamma+1}+\frac{1}{2}\frac{\sigma}{\delta}\int_{\Omega}|\mu^{N}|^{2}.	
\end{equation*}

Taking $j=1$ in~\eqref{Galerkin1} leads to $\partial_{t}\int_{\Omega}n^{N}=\int_{\Omega}\phi_{1}\int_{\Omega}n^{N}G(\mu^{N})\phi_{1}$ where $\phi_{1}$ is constant, and with the assumptions on $G$ (see~\eqref{eq:assumption G}), together with Gronwall's inequality we find 
\begin{equation}
\label{PWG}
\Big|\int_{\Omega}n^{N}\Big|\le C,
\end{equation}
where $C$ is a positive constant independent of $N$.
Using this inequality in~\eqref{Energy-Galerkin} and the assumptions on the source term $G$, we have 
\begin{equation}
\frac{d}{dt}E(t)+\int_{\Omega}B_{\eps}(n^{N})|\nabla\mu^{N}|^{2}\le C.		
\end{equation}
Altogether, we find the following inequalities
 \begin{align}
\frac{\delta}{2}\int_{\Omega}\Big|\nabla\Big(n^{N}-\frac{\sigma}{\delta}\mu^{N}\Big)\Big|^{2}&\le C,\label{Galerkin21}\\
\frac{\sigma}{2\delta}\int_{\Omega}|\mu^{N}|^{2}&\le C,\label{Galerkin22}\\
\int_{\Omega_{T}}B_{\eps}(n^{N})|\nabla\mu^{N}|^{2}&\le C,\label{Galerkin23}\\
\int_{\Omega}\frac{\Big(n^{N}-\frac{\sigma}{\delta}\mu^{N}\Big)^{\gamma+1}}{\gamma+1}&\le C\label{Galerkin26}.
 \end{align}
 
 Using the definition~\eqref{regmob} of $B_{\eps}(n)$ and Inequality~\eqref{Galerkin23} we find a control on $|\nabla\mu^{N}|^{2}$. Combined with Inequality~\eqref{Galerkin21}, we have
 \begin{equation}
 \int_{\Omega_{T}}|\nabla \mu^{N}|^{2}\le C, \quad  \int_{\Omega_{T}}|\nabla n^{N}|^{2}	\le C. \label{Galerkin25}\\
 \end{equation}
 
 Using~\eqref{PWG},~\eqref{Galerkin25} and the Poincaré-Wirtinger inequality, we obtain the following convergence as $N\to +\infty$
\begin{equation}
\label{weakconvergence1}
n^{N}\rightharpoonup n_{\eps}\quad\text{weakly in $L^{2}(0,T,H^{1}(\Omega))$}.
\end{equation}
From \eqref{Galerkin21}-\eqref{Galerkin22},  we find that the coefficients $c_{j}^{N}, d_{j}^{N}$ are bounded and a global solution to~\eqref{initvalue1}--\eqref{initvalue3} exists. Choosing $j=1$ in (\ref{Galerkin2}) gives

\begin{equation*}
\int_{\Omega}\mu^{N}=\int_{\Omega}\Big(n^{N}-\frac{\sigma}{\delta}\mu^{N}\Big)^{\gamma},
\end{equation*}
which is bounded thanks to~\eqref{Galerkin26} and the Hölder inequality. Therefore, combining~\eqref{Galerkin25} and the Poincaré-Wirtinger inequality, we obtain
\begin{equation}
\label{weakconvergence2}
\mu^{N}\rightharpoonup \mu_{\eps}\quad\text{weakly in $L^{2}(0,T,H^{1}(\Omega))$}.
\end{equation} 
%Actually, to prove \eqref{weakconvergence2} it is enough to use inequality \eqref{Galerkin21} and \eqref{Galerkin24}.

Then, denoting by $\Pi_{N}$ the projection operator from $L^{2}(\Omega)$ to $\text{span}\{\phi_{1},...\phi_{N}\}$, using Equation~\eqref{Galerkin1}, we have for every test functions $\phi\in L^{2}(0,T,H^{1}(\Omega))$,
\begin{align*}
\Big|\int_{\Omega_{T}}\partial_{t}n^{N}\phi\Big|&=\Big|\int_{\Omega_{T}}\partial_{t}n^{N}\Pi_{N}\phi\Big|\\
&=\Big|\int_{\Omega_{T}}B_{\eps}(n^{N})\nabla\mu^{N}\nabla\Pi_{N}\phi+\int_{\Omega_{T}}n^{N}G(\mu^{N})\Pi_{N}\phi\Big|\\
&\le C \Big(\int_{\Omega_{T}}B_{\eps}(n^{N})|\nabla\mu^{N}|^{2}\Big)^{1/2}\Big(\int_{\Omega_{T}}|\nabla\Pi_{N}\phi|^{2}\Big)^{1/2}+C\Big(\int_{\Omega_{T}}|\Pi_{N}\phi|^{2}\Big)^{1/2}\\
&\le C\|\phi\|_{L^{2}(0,T,H^{1}(\Omega))},
\end{align*}
where the last inequality is obtain from~\eqref{Galerkin23}.
Therefore, from the previous result, we can extract a subsequence such that 
\begin{equation}
\label{weakconvergence3}
\partial_{t}n^{N}\rightharpoonup	\partial_{t}n_{\eps}\text{weakly in $L^{2}(0,T,H^{-1}(\Omega))$}.
\end{equation}
Hence, combining the weak convergences~\eqref{weakconvergence1} and~\eqref{weakconvergence3}, and using the Lions-Aubin lemma we obtain the strong convergence
\begin{equation}
\label{strongconvergence1}
n^{N}\to  n_{\eps}\quad\text{strongly in $L^{2}(0,T,L^{2}(\Omega))$}.
\end{equation}
\par

\noindent\textit{Step 3. Strong compactness for $\mu^N$.}
To identify the limit in the source term, we need to find strong compactness for the potential $\mu^{N}$. As we have strong compactness for $n^{N}$, it is enough to find strong compactness for $w^{N}$.

Using the notation $w^N = n^N-\f{\sg}{\delta}\mu^N$, we change Equation~\eqref{Galerkin2} in
\[
    \sg  \int_\Omega \nabla w^N \nabla \phi_j + \f\sg\delta\int_\Omega (w^N)^\gamma \phi_j + \int_\Omega w^N \phi_j = \int_\Omega n^N \phi_j,
\]
with $w^N(t,x) = \sum_{i=1}^N q_i^N(t)\phi_i(x)$. The coefficients $q_j^N$ for $j=1,\dots,N$ are determined by the equation 
\begin{equation}
    \left(\sg  \lambda_j + 1 \right)q_j^N + \f\sg\delta\int_\Omega \left(\sum_{i=1}^N q_i^N(t)\phi_i(x)\right)^\gamma \phi_j = c_j^N.
    \label{eq:coeff-wN}
\end{equation}
Thus, denoting $z=\partial_{t}w$, and computing the time derivative of the previous equation, we obtain 
\[
\left( \sg  \lambda_j + 1 \right) \f{d}{dt}q_j^N + \f\sg\delta \f{d}{dt} \int_\Omega \left(\sum_{i=1}^N q_i^N(t)\phi_i(x)\right)^\gamma \phi_j = \f{d}{dt} c_j^N.
\]
We use the notation $\p_t w^N = z^N$, 
Multiplying the previous equation by $\phi_j (\p_t w^N)$, summing over the $j$ and integrating over the domain lead to 
\[
\sg \int_\Omega \abs{\nabla z^N}^2 + \int_\Omega \abs{z^N}^2 + \f{\gamma \, \sg}{\delta}\int_\Omega (w^N)^{\gamma-1}\abs{z^N}^2 = \int_\Omega \p_tn^N z^N.
\]
Integrating the previous equation in time, and from the use of Young's inequality on the right hand side, we have
\begin{equation*}
\int_{\Omega_{T}}\sigma|\nabla z^N|^{2}+\abs{z^N}^{2}\le C\|\partial_{t}n^N\|_{L^{2}(0,T,H^{-1}(\Omega))}^{2}+\frac{\sigma}{2}\|z^N\|_{L^{2}(0,T,H^{1}(\Omega))}^{2}    
\end{equation*}
Then, assuming $1-\f\sg2>0$, we find $\norm{z}_{L^{2}(0,T,H^{1}(\Omega))}^2\le C$. 

This previous result allows us to pass to the limit in the source term using the Lions-Aubin lemma. We refer to \cite{perthame_poulain} to pass to the limit in other terms of the equation and detail the differences. 

Lastly, it remains to prove the regularity of $w$, $w^{\gamma}$. To do so, we use Equation~\eqref{eq:coeff-wN} and successively multiply it by $- \phi_j \Delta w^N $, integrate it in space, use integration by parts, and integrate with respect to time, to arrive to
\begin{equation*}
\sigma\int_{\Omega_{T}}|\Delta w^{N}|^{2}+ \int_\Omega |\nabla w^{N}|^{2}+\frac{\sigma}{\delta}\gamma (w^{N})^{\gamma-1}|\nabla w^{N}|^{2}=-\int_{\Omega_{T}}n^{N}\Delta w^{N}.
\end{equation*}
Since $n^N$ is bounded $L^{2}(0,T,H^{1}(\Omega))$, we can use Young inequality and elliptic regularity to find that $w^N$ is bounded in $L^{2}(0,T,H^{2}(\Omega))$. Inequality~\eqref{Galerkin21} provides the $L^{\infty}(0,T,H^{1}(\Omega))$ bound. Finally the regularity of $(w^{N})^{\gamma}$ comes from Equation~\eqref{eq:KS2}.

This concludes the proof.

\commentout{
The second part of the term in the left hand side is split in two parts, the second one which is positive since $\gamma$ is odd and the first one which is
\begin{equation*}
\gamma(\gamma-1)\int_{\Omega_{T}}\Delta w^{N} (w^{N})^{\gamma-2}|\nabla w^{N}|^{2}.
\end{equation*}
 Using Hölder inequality for the three functions we can bound uniformly this last term. This provides the bound on $\nabla\Delta w$ in $L^{2}(0,T,L^{2}(\Omega))$. We conclude by elliptic regularity. 
\textcolor{red}{Do we have $w\in H^{2}$ implies $w$ in every $L^{p}$ in dimension $d\le 4$? Or just $d\le 3$}
\textcolor{red}{Since $w\in H^{3}$ we even have $w\in W^{1,\infty}$ if $d\le 3$}.
}

\end{proof}

\commentout{
\subsection{Nonnegativity of the solutions}
\textcolor{red}{À enlever}
\begin{proposition}
Let $(n,\mu)$ be weak solutions  to~\eqref{eq:CH1} and~\eqref{eq:CH2} in a sense defined later on with an initial condition $n^{0}\ge 0$. Then, $n(t)$ is nonnegative for every $t$.
\end{proposition}

\begin{proof}
We define $|n|_{-}=-n_{-}=\max(0,-n)$  and we multiply Equation~\eqref{eq:CH1} by $\mathbbm{1}_{\{n<0\}}$.
This yields
\begin{equation*}
\frac{d}{dt}|n|_{-}-\Div(|n|_{-}\nabla \mu)=|n|_{-}G(\mu).
\end{equation*}
Integrating in space, using the boundary conditions and the assumption on $G$ gives
\begin{equation*}
\frac{d}{dt}\int_{\Omega}|n|_{-}\le C\int_{\Omega}|n|_{-}
\end{equation*}
With Grönwall's lemma and the assumption on the initial condition, we obtain the result.
\end{proof}

\begin{proposition}
Let $w$ be a weak solution to ..... Then, $w\ge 0$.	
\end{proposition}

\begin{proof}
From the definition of $w$ we have $-\Delta w+w^{\gamma}+w=n$ where we supposed $\delta=\sigma=1$ for sake of clarity.\\
Since $n\ge 0$, this definition implies that for every function $\eta\ge 0$, we have
\begin{equation*}
\int_{\Omega}\nabla w\cdot\nabla\eta +\int_{\Omega}w^{\gamma}\eta+\int_{\Omega}w\eta\ge 0.	
\end{equation*}
Using this inequality with $\eta=-w_{-}$ we find
\begin{equation*}
\int_{\Omega}|\nabla w_{-}|^{2}+\int_{\Omega}w_{-}^{2}\le-\int_{\Omega}w_{-}^{\gamma+1}.
\end{equation*}
The term of the right hand side is negative since $\gamma+1$ is even. This yields the result.
\end{proof}
}
\subsection{A priori estimates} \label{sec:apriori-eps}
To show the existence of weak solutions of the non-regularized model~\eqref{eq:CH1}--\eqref{eq:CH2}, the idea is to pass to the limit $\eps \to 0$ in System~\eqref{CHR1}--\eqref{CHR2}. 
A priori estimates derived from the regularized model help us to find the required compacity and pass to the limit in the model.
The computations of the a priori estimates follow closely the paper \cite{perthame_poulain} where the case of a single-well potential was considered. The addition of the source term $G$, which is not present in \cite{perthame_poulain}, induces the need of new computations. The second main difference is that we consider a potential term $w^{\gamma}$ in~\eqref{eq:CH2} instead of a smooth and bounded potential as it was done before.

We start by defining the entropy of the system 
\begin{equation*}
\Phi_{\eps}[n]=\int_{\Omega}\phi_{\eps}(n(x))dx,\quad \phi_{\eps}''(x)=\frac{1}{B_{\eps}(n)},\quad\phi_{\eps}(0)=\phi_{\eps}'(0)=0,
\end{equation*}
and we recall the energy

\begin{equation*}
\mathcal{E}_{\eps}[n]=\int_{\Omega}\frac{(n-\frac{\sigma}{\delta}\mu)^{\gamma+1}}{\gamma+1}+\frac{\delta}{2}\Big|\nabla \Big(n-\frac{\sigma}{\delta}\mu\Big)\Big|^{2}+\frac{\sigma}{\delta}\frac{|\mu|^{2}}{2}.
\end{equation*}

\commentout{
\begin{proposition}[Improved regularity]
We have in fact
\begin{align*}
n&\in L^{\infty}(\Omega_{T}),\quad \partial_{t}n\in L^{2}(0,T,H^{-1}(\Omega)),\\
\partial_{t}\mu&\in L^{2}(0,T,H^{-1}(\Omega)),\\
\partial_{t}w&\in L^{2}(0,T,H^{1}(\Omega)).
\end{align*}
\end{proposition}
We refer the reader to the following section for the proof. 
}

In comparison with the previous subsection, when $\eps\to 0$ we loose the uniform $L^{2}(\Omega_{T})$ bound on $\nabla\mu$ which was obtained with the regularized mobility. This bound is actually recovered with the estimates provided by the entropy that could not be used in the Galerkin scheme. Also,  when $\eps\to 0$, the time derivative of the density lies in a larger space than $L^{2}(0,T,H^{-1}(\Omega))$, which is $(L^{4}(0,T,W^{1,s}(\Omega)))'$ where $s>2$. This prevents us from using the previous computations and we lose the bound on $\partial_{t} w$ and $\partial_{t}\mu$. To recover their strong convergence and thus to identify the source term, we need to rely on the Fréchet-Kolmogorov theorem. This last difficulty is dealt with the following proposition 

\begin{proposition}[Compactness for $w_\eps$ and $\mu_\eps$]
\label{compactnessformu}
The sequences $(w_{\eps})_{\eps}$ and $(\mu_{\eps})_{\eps}$ converge strongly in $L^{1}(\Omega_{T})$.
\end{proposition}

\begin{proof}
We use the Fréchet-Kolmogorov theorem as in Appendix~\ref{app:FK}. Indeed, we have the compactness in space for $w$ and Equation~\eqref{eq:KS2} provides
\begin{equation*}
-\sigma\Delta [w(t+h)-w(t)]+\frac{\sigma}{\delta}[w^{\gamma}(t+h)-w^{\gamma}(t)]+	w(t+h)-w(t)=n(t+h)-n(t).
\end{equation*}
Multiplying formally by $\sgn (w(t+h)-w(t))$ (rigorously by $\phi'(w(t+h)-w(t))$ where $\phi$ is a convex approximation of the absolute value), integrating in space and using Appendix~\ref{app:FK}, we obtain the compactness for $w$. Since $\mu=\frac{\delta}{\sigma}(n-w)$ we have compactness for $\mu$.

\end{proof}

\begin{remark}
The strong convergence actually holds in higher Lebesgue spaces since $(w_{\eps})_{\eps}$ is bounded uniformly in $L^{2}(0,T,H^{2}(\Omega))\cap L^{\infty}(0,T,L^{\gamma+1}(\Omega))$ and $(\mu_{\eps})_{\eps}$ in $L^{2}(0,T,H^{1}(\Omega))$. 	
\end{remark}

\subsection{Limit \texorpdfstring{$\eps \to 0$}{}}
From the previous estimates on the regularized model, we are now in position to prove the existence of global weak solutions for System~\eqref{eq:CH1}--\eqref{eq:CH2}.
\begin{thm}[Existence of weak solutions] \label{th:existence-sol}
Assume an initial condition satisfying~\eqref{initcond}. Then, for $\sigma$ small enough, there exists a global weak solutions $(n,\mu)$ of Equations~\eqref{eq:CH1}-\eqref{eq:CH2} such that
\begin{align*}
n&\in L^{2}(0,T,H^{1}(\Omega)),\quad\partial_{t}n\in (L^{4}(0,T,W^{1,s}(\Omega)))',\\
\mu&\in L^{2}(0,T,H^{1}(\Omega)),\\
w&\in L^{2}(0,T,H^{2}(\Omega))\cap L^{\infty}(0,T,H^{1}(\Omega)),\quad w^{\gamma}\in L^{2}(0,T,L^{2}(\Omega)),\\
n&\ge 0,\quad  w\ge 0,\quad\text{a.e. in $\Omega_{T}$},
\end{align*}
where $s$ is defined in proposition~\ref{compactnessh2}. Moreover, as $\eps\to 0$, the inequalities provided by the energy and entropy hold true.
\end{thm}
\begin{proof}
The estimates on $\partial_{t}n$ has been proved in Section 2. The proof of this theorem is a straightforward adaptation of Theorem 5 in \cite{perthame_poulain} using the computation of Section~\ref{sec:apriori-eps}, therefore, we do not repeat the proof arguments here. The nonnegativities of $n$ and $w$ are a consequence of Proposition~\ref{nonnegativities}. 
\end{proof}

The regularity of the solutions is higher than it is expected in Theorem~\ref{th:existence-sol}. We refer to Section~\ref{sec:inpressible-limit} for the proof of this result.

\begin{proposition}[Nonnegativity of $n$ and $w$]
\label{nonnegativities}
The solution $(n,w)$ of System~\eqref{eq:KS1}--\eqref{eq:KS2} defined by Theorem~\ref{th:existence-sol} is nonnegative, \ie 
\[
    n(x,t),\, w(x,t)\ge 0, \quad \text{a.e. in } \Omega.   
\]
\end{proposition}

\begin{proof}
The proof of the nonnegativity of the density $n$ follows the same argument than \cite{perthame_poulain} that uses the boundedness of the entropy (uniform in $\eps$). Hence, we do not repeat the proof here.   
Now, we recall that the repulsive potential from~\eqref{eq:KS2} is actually defined by assumption in the introduction as $w^{\gamma}= \max(0,w^{\gamma})$.
Then, since $n\ge 0$, Equation~\eqref{eq:KS2} implies that for every function $\eta\ge 0$, we have 
\begin{equation*}
\int_{\Omega}\nabla w\cdot\nabla\eta +\int_{\Omega}\max(0,w^{\gamma})\eta+\int_{\Omega}w\eta\ge 0,	
\end{equation*}
where we supposed $\sigma=\delta=1$ for the sake of clarity.
Using the previous inequality, and choosing $\eta=-w_{-}$ (where $w_{-}$ represents the negative part of $w$), we obtain
\begin{equation*}
\int_{\Omega}|\nabla w_{-}|^{2}+\int_{\Omega}w_{-}^{2}\le 0.
\end{equation*}
This achieves the proof of the nonnegativity of $w$.   
\end{proof}

\begin{remark}
A main difference with \cite{perthame_poulain} is that we can not find an upper bound using an entropy argument. Indeed, the result in~\cite{perthame_poulain} relies on the singularity of the potential at $n=1$. Furthermore, Dai and Du~\cite{Dai-Du} notice that with a smooth potential at $n=1$, one cannot prevent the solution from exceeding this threshold. This is due to the Gibbs-Thomson effect, where the mean curvature of the interface plays a role in the concentration of the phases. 

However, from a different argument relying on the Alikakos iteration method~\cite{Alikakos}, see Inequality~\eqref{ine32}, we are able to show an $L^\infty$-bound for System~\eqref{eq:KS1}--\eqref{eq:KS2} and therefore better regularity for the functions of our system.
\end{remark}

\begin{comment}
We conclude with a uniqueness theorem
\begin{thm}[Uniqueness without a source term]
\label{thm:uniqueness}
The weak solutions of the system~\eqref{eq:CH1}--\eqref{eq:CH2} with no source term (i.e $G=0$) are unique.
\end{thm}

It has been shown in \cite{Sugiyama_uniqueness,Carrillo} that solutions to many Keller-Segel systems are unique. However the system considered by the authors contains no source term. In the first article, the authors prove uniqueness of solution of the Keller-Segel model in the class of Hölder continuous functions using the duality method. Their argument cannot be applied in our case as the equation \eqref{eq:KS2} on $w$ is nonlinear. In the second article the authors consider arguments from optimal transport using the fact that $\nabla w$ is log-lipschitz to prove uniqueness. We follow the second article to prove uniqueness for our system, considering that our source term is 0.

\begin{proof}
The proof is based on~\cite{Carrillo} Section 6. Using proposition~\ref{higherregularityforw} we find that $\nabla w$ is log-lipschitz. One can then prove that two solutions $w_{0},w_{1}$ of \eqref{eq:KS2} with functions $n_{0},n_{1}$ respectively satisfy 
\begin{equation*}
\|w_{1}-w_{0}\|_{L^{\infty}(0,T,H^{1}(\Omega))}\le \|n_{1}-n_{0}\|_{L^{\infty}(0,T,H^{-1}(\Omega))}.	
\end{equation*}
With these estimates the proof follows the lines of~\cite{Carrillo}.
\end{proof}
\end{comment}

\appendix
%------------------------------------------------------

\section{Compactness with the Fréchet-Kolmogorov theorem}
\label{app:FK}
%------------------------------------------------------

We provide another method to prove strong compactness in $\gamma$ for $n$ only using the fact that $n\nabla\mu$ and $\nabla n$ are integrable. This proposition can be applied to prove compactness in $\sigma$ and $\eps$ with the parameters defined respectively in Section 2 and 4.\\

\begin{proposition}[Strong compactness for $n$]
\label{compactness}
The sequence $(n_{\sigma,\gamma})_{\gamma}$ is compact in $L^{2}(0,T,L^{p}(\Omega))$ for $1\le p<2d/(d-2)$ if $d> 2$ and $1\le p<\infty$ else.
\end{proposition}

The proof of this proposition uses a sequence $(\varphi_{\delta})_{\delta>0}\in C_{c}(\R^{d})$ of standard mollifiers with mass~1 such that
\begin{equation*}
 \|\nabla^{k}\varphi_{\delta}\|_{L^{1}(\Omega)}\le\frac{C}{\delta^{k}},
\end{equation*}
for any function $g\in L^{p}(\Omega)$,
\begin{equation*}
\|g\ast\varphi_{\delta}\|_{L^{p}(\Omega)}\le \|\varphi_{\delta}\|_{L^{1}(\Omega)}\|g\|_{L^{p}(\Omega)},	
\end{equation*}
and when $g\in W^{1,p}(\Omega)$ it holds
\begin{equation*}
\|g\ast\varphi_{\delta}-g\|_{L^{p}}\le\delta\|\nabla g\|_{L^{p}(\Omega)}.	
\end{equation*}

\begin{proof}[Proof of proposition \ref{compactness}]
Since  $\nabla n$ is bounded in $L^{1}(\Omega_{T})$ we only need to prove the time compactness
\begin{equation*}
\lim_{|h|\to 0}\int_{0}^{T-h}\int_{\Omega}|n_{\sigma,\gamma}(t+h,x)-n_{\sigma,\gamma}(t,x)|dxdt=0\quad \text{uniformly in $\gamma$}	.
\end{equation*}
Using the mollifiers with $\delta$ depending on $h$ to be specified later on, we first notice that
\begin{align*}
\int_{0}^{T-h}\int_{\Omega}|n(t+h,x)-n(t,x)|dxdt&\le \int_{0}^{T-h}\int_{\Omega}|n(t,x)-n(t,\cdot)\ast\varphi_{\delta}(x)|dxdt\\
&+\int_{0}^{T-h}\int_{\Omega}|n(t+h,x)-n(t+h,\cdot)\ast\varphi_{\delta}(x)|dxdt\\
&+\int_{0}^{T-h}\int_{\Omega}|n(t+h,\cdot)\ast\varphi_{\delta}(x)-n(t,\cdot)\ast\varphi_{\delta}(x)|dxdt.
\end{align*}
For the first and second term, the computations are the same, hence, we only present it for the first term. Using the properties of the mollifiers, we have
\begin{equation*}
\int_{0}^{T-h}\int_{\Omega}|n(t,x)-n(t,\cdot)\ast\varphi_{\delta}(x)|dxdt\le \delta\int_{0}^{T}\|\nabla n(t)\|_{L^{1}(\Omega)}dt\le C\delta.
\end{equation*}
To control $\|\nabla n\|_{L^{1}(\Omega)}$ we used the Cauchy-Schwarz inequality and Proposition~\ref{compactnessh2}.
The third term can be written as 
\begin{equation*}
\int_{0}^{T-h}\int_{\Omega}|n(t+h,\cdot)\ast\varphi_{\delta}(x)-n(t,\cdot)\ast\varphi_{\delta}(x)|dxdt=\int_{0}^{T-h}\int_{\Omega}\Big|\int_{t}^{t+h}\partial_{s}n(s,x)\ast\varphi_{\delta}(x)ds\Big|dxdt.
\end{equation*}
And using Equation~\eqref{eq:CH1} yields
\begin{align*}
\int_{0}^{T-h}\int_{\Omega}\Big|\int_{t}^{t+h}\partial_{s}n(s,\cdot)&\ast\varphi_{\delta}(x)ds\Big|dxdt \\
&=\int_{0}^{T-h}\int_{\Omega}\Big|\int_{t}^{t+h}[\Div(n\nabla\mu)+nG(\mu)](s,\cdot)\ast\varphi_{\delta}(x)ds\Big|dxdt\\
&\le\int_{0}^{T-h}\int_{\Omega}\Big|\int_{t}^{t+h}\sum_{i=1}^{d}n\partial_{x_{i}}\mu(s,\cdot)\ast\partial_{x_{i}}\varphi_{\delta}(x)ds\Big|dxdt\\
&+\int_{0}^{T-h}\int_{\Omega}\Big|\int_{t}^{t+h}nG(\mu)(s,\cdot)\ast\varphi_{\delta}(x)ds\Big|dxdt.
\end{align*}
The first term is bounded by 
\begin{equation*}
\int_{0}^{T-h}\int_{\Omega}\Big|\int_{t}^{t+h}\sum_{i=1}^{d}n\partial_{x_{i}}\mu(s,\cdot)\ast\partial_{x_{i}}\varphi_{\delta}(x)ds\Big|dxdt\le C\int_{0}^{T-h}\int_{t}^{t+h}\|n\nabla\mu(s)\|_{L^{1}(\Omega)}	\|\nabla\varphi_{\delta}\|_{L^{1}(\Omega)}.
\end{equation*}

Writing $n\nabla\mu=n^{1/2}n^{1/2}\nabla\mu$, using the Cauchy-Schwarz inequality, Inequality \eqref{ine511}, the first term is bounded by $Ch/\delta$.\\
The second term is bounded by $Ch$ using assumptions on $G$. Choosing $\delta=h^{1/2}$ gives compactness in $L^{1}(\Omega_{T})$.\\
Finally using Proposition~\ref{compactnessh2} and interpolation, we get the result.

\end{proof}

\section{Uniqueness with no source term}
\label{app:uniqueness}
We consider Equations~\eqref{eq:KS1}-\eqref{eq:KS2} with $G=0$. Therefore we have conservation of the mass.  To simplify the notations we suppose $\sigma=\delta=1$. We retain the regularity of the solutions from     Section~\ref{sec:inpressible-limit}. Writing $n=n_{2}-n_{1}$ and $w=w_{2}-w_{1}$ the difference of two solutions we consider $\varphi(t)$ such that
\begin{align*}
    -\Delta\varphi(t)&=n\quad\text{in $\Omega$},\\
    \nabla\varphi(t)\cdot\nu&=0\quad\text{on $\partial\Omega$}.
\end{align*}
Multiplying Equation~\eqref{eq:KS1} by $\varphi$, integrating in space and using integrations by parts we obtain 
\begin{align*}
\frac{1}{2}\frac{d}{dt}\int|\nabla\phi|^{2}-\int (n_{2}^{2}-n_{1}^{2})\Delta\varphi-\int n\nabla w_{2}\cdot\nabla\varphi-\int n_{1}\nabla w\cdot\nabla\varphi&=0.
\end{align*}
We denote by $I_{1}, I_{2}, I_{3}$ the last three terms. 
Using the definition of $\varphi$, we have
\begin{equation*}
I_{1}=-\int (n_{2}^{2}-n_{1}^{2})\Delta\varphi=\int (n_{2}^{2}-n_{1}^{2})(n_{2}-n_{1}),
\end{equation*}
which is nonnegative since $x\mapsto x^{2}$ is increasing for $x\ge 0$.
Now, we treat $I_{2}$. Using integration by parts and $n=-\Delta\varphi$ one can show that 

\begin{align*}
I_{2}&=-\sum_{i,j}\int\partial_{i}\varphi\partial_{ij}w_{2}\partial_{j}\varphi-\sum_{i,j}\int\partial_{i}\varphi\partial_{j}w_{2}\partial_{ij}\varphi\\
&=J_{1}+J_{2}.
\end{align*}
Using integrations by parts on the second term we find
\begin{align*}
J_{2}&=\sum_{i,j}\int\partial_{j}\frac{|\partial_{i}\varphi|^{2}}{2}\partial_{j}w_{2}+\sum_{i,j}\int|\partial_{i}\varphi|^{2}\partial_{jj}w_{2}\\
&=\frac{1}{2}\int|\nabla\varphi|^{2}\Delta w_{2}.
\end{align*}
We finally obtain
\begin{equation*}
I_{2}\le C\int |D^{2}w_{2}||\nabla\varphi|^{2}.
\end{equation*}
From Inequality~\eqref{ine33} together with the
Calderón–Zygmund lemma and using $\nabla\varphi\in L^{\infty}$ we find
\begin{align*}
I_{2}&\le Cp\Big(\int|\nabla\varphi|^{2p/(p-1)}\Big)^{(p-1)/p}\\
&\le Cp\norm{\nabla\varphi}_{L^{\infty}}^{2/p}\Big(\int|\nabla\varphi|^{2}\Big)^{(p-1)/p},
\end{align*}
for every $1\le p<\infty$.
For $I_{3}$ we recall that $n_{1}$ is bounded in $L^{\infty}$. Using the Cauchy-Schwarz inequality we only need to show that $\norm{\nabla w}_{L^{2}}\le C\norm{\nabla\varphi}_{L^{2}}$.

Equation~\eqref{eq:KS2} for the difference of the two solutions can be written as
\begin{equation*}
-\Delta w+w_{2}^{\gamma}-w_{1}^{\gamma}+w=-\Delta\varphi.
\end{equation*}
Now, we multiply by $w$, use integration by parts and the fact that $x\mapsto x^{\gamma}$ is increasing for $x\ge 0$ to get 
\begin{equation*}
\int |\nabla w|^{2}+w^{2}\le \int |\nabla\varphi\cdot\nabla w|.   
\end{equation*}
Applying Young's inequality yields the result. 
Therefore 
\begin{equation*}
I_{3}\le C\norm{\nabla\varphi}_{L^{2}}^{2}.
\end{equation*}
Combining the previous results we obtain 
\begin{equation*}
\frac{d}{dt}\eta(t)\le Cp\max(\eta(t)^{1-1/p},\eta(t)),    
\end{equation*}
for every $1\le p<\infty$ where $\eta(t)=\int |\nabla\varphi(t)|^{2}$. 

Following~\cite{Bertozzi-2010,Bedrossian_2011}, for $t<1/C$, the solution is bounded by
\begin{equation*}
 \eta(t)\le (Ct)^{p}.  
\end{equation*}
For $t<1/(2C)$ we then find
\begin{equation*}
\eta(t)\le 2^{-p}.
\end{equation*}
Letting $p\to\infty$ we find uniqueness on $[0,\frac{1}{2C})$ and this procedure can be iterated on the whole interval of existence.

\bibliographystyle{siam}
\bibliography{biblio}

\end{document}